\documentclass[12pt,reqno]{amsart}
\usepackage{amsthm, amsmath, amssymb, amscd, latexsym, multicol, verbatim, enumerate}
\usepackage[top=1.13in, bottom=1.13in, left=1.15in, right=1.15in]{geometry}
\usepackage[colorlinks=true, pdfstartview=FitV, linkcolor=blue, citecolor=blue, urlcolor=blue, breaklinks=true]{hyperref}

\usepackage[english]{babel}
\usepackage{mathrsfs}
\usepackage[all]{xy}
\usepackage[dvips]{graphicx}

\newcounter{cnt}
 \makeatletter
\def\mydggeometry{\makeatletter\dg@YGRID=1\dg@XGRID=20\unitlength=0.003pt\makeatother}
\makeatother \theoremstyle{remark}
\numberwithin{equation}{section}

\theoremstyle{definition} 
\newtheorem{definition}{Definition}\theoremstyle{definition}
\newtheorem{proposition}{Proposition}
\newtheorem{theorem}{Theorem}
\newtheorem{lemma}{Lemma}
\newtheorem*{lemma*}{Lemma}
\newtheorem{corollary}{Corollary}
\newtheorem{remark}{Remark}
\newtheorem{example}{Example}
\newtheorem{conjecture}{Conjecture}

\newtheorem*{conjecture*}{Conjecture}
\newtheorem*{thma}{Theorem A}
\newtheorem*{thmb}{Theorem B}

\newcommand{\g}{\mathfrak{g}}
\newcommand{\n}{\mathfrak{n}}

\newcommand{\ra}{\rightarrow}

\newcommand{\vp}{\varphi}
\newcommand{\ts}{\otimes}

\newcommand{\ff}{\mathcal{F}}

\newcommand{\mc}{\mathbb{C}}
\newcommand{\md}{\mathcal{D}}
\newcommand{\ms}{\mathcal{S}}

\newcommand{\gr}{\operatorname*{gr}}

\newcommand{\lie}{\mathfrak}
\newcommand{\ddeg}{\operatorname*{deg}}

\newcommand{\sspan}{\operatorname*{span}}

\newcommand{\diag}{\operatorname*{diag}}
\newcommand{\Mat}{\operatorname*{Mat}}

\newcommand{\lm}{\operatorname*{lm}}
\newcommand{\gm}{\operatorname*{gm}}

\newcommand{\bs}{\mathbf{s}}
\newcommand{\bd}{\mathbf{d}}
\newcommand{\bt}{\mathbf{t}}
\newcommand{\Dwq}{\md_{\underline{w}_0}^q}
\newcommand{\N}{\mathbb{N}}

\begin{document}


\title{Degree cones and monomial bases of Lie algebras and quantum groups}

\author{Teodor Backhaus, Xin Fang and Ghislain Fourier}
\address{\newline T.B., X.F., Mathematisches Institut, Universit\"at zu K\"oln, Germany} 
\email{tbackha@math.uni-koeln.de}
\email{xinfang.math@gmail.com}
\address{\newline G.F., School of Mathematics and Statistics, University of Glasgow, UK}
\email{ghislain.fourier@glasgow.ac.uk}

\begin{abstract}
We provide $\mathbb{N}$-filtrations on the negative part $U_q(\mathfrak{n}^-)$ of the quantum group associated to a finite-dimensional simple Lie algebra $\mathfrak{g}$, such that the associated graded algebra is a skew-polynomial algebra on $\mathfrak{n}^-$. The filtration is obtained by assigning degrees to Lusztig's quantum PBW root vectors. The possible degrees can be described as lattice points in certain polyhedral cones. In the classical limit, such a degree induces an $\mathbb{N}$-filtration on any finite dimensional simple $\mathfrak{g}$-module. We prove for type $\tt{A}_n$, $\tt{C}_n$, $\tt{B}_3$, $\tt{D}_4$ and $\tt{G}_2$ that a degree can be chosen such that the associated graded modules are defined by monomial ideals, and conjecture that this is true for any $\mathfrak{g}$.
\end{abstract}

\maketitle
\section{Introduction}

\subsection{Motivation}

Let $\g$ be a finite dimensional simple Lie algebra with a triangular decomposition $\g=\mathfrak{n}^+\oplus\mathfrak{h}\oplus\mathfrak{n}^-$. Let $U(\lie n^-)$ be the corresponding universal enveloping algebra. Lusztig \cite{Lus1} introduced the canonical basis for the quantized enveloping algebra $U_q(\lie n^-)$. Subsequently, Kashiwara \cite{Kas} gave a different construction of this basis under the name global crystal basis. When the quantum parameter is specialized to $1$, the canonical basis is specialized to a linear basis $\mathcal{B}$ of $U(\lie n^-)$. Let $V(\lambda)$ be the finite dimensional irreducible representation of $U(\g)$ of highest weight $\lambda$ and $v_\lambda$ be a fixed highest weight vector. The canonical basis $\mathcal{B}$ of $U(\lie n^-)$ induces a canonical basis of $V(\lambda)$ by:
$$\mathcal{B}_\lambda=\{b\cdot v_\lambda\mid b\in\mathcal{B},\ b\cdot v_\lambda\neq 0\}.$$
This important property distinguishes the canonical basis from other bases of $U(\lie n^-)$ and $V(\lambda)$.
\par
In this paper, motivated by \cite{FFR15, FFL1, FFL2}, we are interested in the existence of monomial bases $\mathcal{E}$ of $U(\lie n^-)$ satisfying the following properties:
\begin{enumerate}[(P1)]
\item there exists an $\mathbb{N}$-filtration $\ff$ on $U(\lie n^-)$ such that the associated graded algebra is the polynomial algebra $S(\lie n^-)$; the set $\mathcal{E}$ is a linear basis of the associated graded algebra;
\item let $V^\ff(\lambda)$ be the associated graded $S(\lie n^-)$-module with cyclic vector $v_\lambda^\ff$, 
$$\mathcal{E}_\lambda=\{b\cdot v_\lambda^\ff\mid b\in\mathcal{E},\ b\cdot v_\lambda^\ff\neq 0\}$$
is a linear basis of $V^\ff(\lambda)$, hence a linear basis of $V(\lambda)$.
\end{enumerate}
An equivalent formulation of (P2) in terms of monomial ideals is:
\begin{enumerate}[(P2')]
\item find an $\mathbb{N}$-filtration $\ff$ on $U(\lie n^-)$ such that for any dominant integral weight $\lambda$, the defining ideal of $\operatorname{gr}^{\mathcal{F}}V(\lambda)$ is monomial.
\end{enumerate}

By turning back to quantum groups, we may ask similar questions:
\begin{enumerate}[(Q1)]
\item Is there an $\mathbb{N}$-filtration of $U_q(\lie n^-)$ such that the associated graded algebra is $S_q(\lie n^-)$ (a skew-polynomial algebra on the vector space $\lie n^-$).
\end{enumerate}

\subsection{Answering (Q1)}
The answer to (P1) is rather trivial. We fix a weighted basis of $\lie n^-$, indexed by positive roots $\Delta^+$, and let $\mathbf{d}: \mathfrak{n}^- \longrightarrow \mathbb{N}$ be a degree function on $\lie{n}^-$ such that for any basis elements $x,y \in \lie n^-$:  \[ 
\mathbf{d}(x) + \mathbf{d}(y) > \mathbf{d}([x,y]).\] This induces an $\mathbb{N}$-filtration on $U(\lie n^-)$ and the induced associated graded algebra is isomorphic to $S(\lie n^-)$. We denote $\mathcal{D}$, called the \emph{classical degree cone}, the real cone generated by all degree functions on $\mathfrak{n}^-$ (respectively $\Delta^+$) satisfying these inequalities.
\par
To construct an $\mathbb{N}$-filtration on $U_q(\lie n^-)$, it is not enough to consider its Chevalley generators $F_1, \ldots, F_n$, since $U_q(\lie n^-)$ is already a graded algebra for any grading on these generators, and the defining ideals of simple modules are seldom monomial.
\par
There is another basis of $U_q(\lie n^-)$ given by Lusztig \cite{Lus2}, called quantum PBW basis. Let $\underline{w}_0=s_{i_1} \ldots s_{i_N}$ be a reduced decomposition of the longest Weyl group element. We associate a sequence of elements $F_{\beta_1}, \ldots, F_{\beta_N} \in U_q(\lie n^-)$, where $\{ \beta_1, \ldots, \beta_N\}$ is the set of positive roots and $F_{\beta_i}$ is a quantum PBW root vector of weight $-\beta_i$. 
Then Lusztig has shown that ordered monomials in $F_{\beta_1}, \ldots, F_{\beta_N}$ form a linear basis of $U_q(\lie n^-)$.
\par
The naive approach of setting the degree of $F_{\beta_i}$ to $1$ for all $\beta_i$ does not provide  $\operatorname{gr } U_q(\lie n^-) \cong S_q(\lie n^-)$ for the induced filtration:

\begin{example}
 Let $\g=\mathfrak{sl}_4$ be of type $\tt A_3$. Fix the reduced decomposition $\underline{w}_0=s_1s_2s_1s_3s_2s_1$ of the longest element $w_0$ in the Weyl group of $\g$. We denote by $F_{i \ i+1\cdots k}$ the quantum PBW root vector corresponding to the root $-(\alpha_i+\alpha_{i+1}+\cdots+\alpha_k)$. The following relation holds in $U_q(\mathfrak{n}^-)$:
$$F_{23}F_{12}=F_{12}F_{23}-(q-q^{-1})F_2F_{123}.$$
\end{example}

In general, the commutation relations in $U_q(\lie n^-)$ are given by the following Levendorskii-Soibelman (L-S for short) formula: for any $i<j$: 
\begin{equation*}
F_{\beta_j}F_{\beta_i}-q^{-(\beta_i,\beta_j)}F_{\beta_i}F_{\beta_j}=\sum_{n_{i+1},\ldots,n_{j-1}\geq 0}c(n_{i+1},\ldots,n_{j-1})F_{\beta_{i+1}}^{n_{i+1}}\ldots F_{\beta_{j-1}}^{n_{j-1}}.
\end{equation*}
These commutation relations depend heavily on the choice of reduced decomposition $\underline{w}_0$. For a given reduced decomposition $\underline{w}_0$, we seek for degree functions on the set of positive roots
\[
\mathbf{d}: \Delta_+ \longrightarrow \mathbb{N}
\]
such that letting $\deg(F_\beta)=\bold{d}(\beta)$ for $\beta\in\Delta_+$ defines a filtered algebra structure on $U_q(\lie n^-)$ and the associated graded algebra satisfies: $\operatorname{gr}^\mathbf{d} U_q(\lie n^-) \cong S_q(\lie n^-)$. Inspired by the L-S formula, we define for any reduced decomposition $\underline{w}_0$ the \emph{quantum degree cone} $\md^q_{\underline{w}_0}$ by:
\[
\md^q_{\underline{w}_0}:=\{(d_\beta)\in\mathbb{R}_+^{|\Delta_+|}\ |\ \text{for any } i<j, d_{\beta_i}+d_{\beta_j}>\sum_{k=i+1}^{j-1}n_kd_{\beta_k} \text{ if } c(n_{i+1},\ldots,n_{j-1})\neq 0\}.
\]
The first main theorem of this paper is:
\begin{thma}
Let $\underline{w}_0$ be a reduced decomposition. Then
\begin{enumerate}
\item the set $\md_{\underline{w}_0}^q$ is a non-empty, open polyhedral cone;
\item a degree function $\mathbf{d}: \Delta^+ \longrightarrow \mathbb{N}$ defines a filtered algebra structure on $U_q(\lie n^-)$ such that $\operatorname{gr}^\mathbf{d} U_q(\lie n^-) \cong S_q(\lie n^-)$ if and only if $\mathbf{d} \in \md^q_{\underline{w}_0} \cap \mathbb{Z}^N.$
\end{enumerate}
\end{thma}

It is natural to ask whether there is a uniform degree function $\mathbf{d}$ which is compatible with every reduced decomposition. We will show that for a simple Lie algebra $\lie g$, such a function exists if and only if the rank of $\lie g$ is less or equal than $2$, i.e., for any $\lie g$ of rank larger than $2$, we have
$$ \bigcap_{\underline{w}_0\in R(w_0)} \mathcal{D}_{\underline{w}_0}^q = \emptyset,$$
where $R(w_0)$ is the set of all reduced decompositions of $w_0$.
\par
Suppose $\lie g$ is of simply-laced type and the reduced decomposition is adapted to an orientation of the associated Dynkin quiver. Using the Hall algebra realization of $U_q(\lie n^-)$, the coordinates of the lattice points in the quantum degree cone have an interpretation as dimensions of certain homomorphism spaces for the particular Dynkin quiver \cite{FFR15}. 

\subsection{Answering (P2')}
We turn from the quantum situation to the classical one and analyze the implication of the induced filtration for finite-dimensional simple modules.
\par
Let $V(\lambda)$ be the simple module of highest weight $\lambda$. Since $V(\lambda) = U(\lie n^-)\cdot v_\lambda$, any filtration on $U(\lie n^-)$ induces a filtration on $V(\lambda)$.
\par
Let $\mathbf{d}:\Delta_+\ra\mathbb{N}$ be a degree function for $U(\lie n^-)$ such that $\operatorname{gr}^\mathbf{d} U(\lie n^-) \cong S(\lie n^-)$. The associated graded module $\operatorname{gr}^\mathbf{d} V(\lambda)$ of the induced filtration is a  cyclic $ S(\lie n^-)$-module. Hence there exists an ideal $I^\bold{d}(\lambda) \subset S(\lie n^-)$ such that $\operatorname{gr}^\mathbf{d} V(\lambda) \cong S(\lie n^-)/I^\bold{d}(\lambda)$.
\par
Our second aim of the paper is to find monomial bases of $\operatorname{gr}^\mathbf{d} V(\lambda)$. If the ideal $I^\bold{d}(\lambda)$ is monomial, there exists a unique monomial basis for $\operatorname{gr}^\mathbf{d} V(\lambda)$. We will focus on this case in the paper. 
\par
The \textit{global monomial set} $\mathcal{S}_{\gm}$ consists of all degree functions $\bold{d}:\Delta_+ \ra \mathbb{N}$ such that for any dominant integral weight $\lambda$, $I^\bold{d}(\lambda)$ is a monomial ideal. As the other main result of this paper, all monomial bases appearing in the context of PBW filtration in the literature can be actually obtained through a degree in the global monomial set. 
\begin{thmb}
Let $\g$ be a simple Lie algebra of type $\tt A_n$, $\tt C_n$, $\tt B_3$, $\tt D_4$ or $\tt G_2$. Then $\ms_{\gm}\neq\emptyset$.
\end{thmb}
We provide a degree function in the global monomial set in each case (for the $\tt A_n$-case this has been done already in \cite{FFR15}). Based on the evidence of several further examples, we conjecture:
\begin{conjecture*}
\begin{enumerate}
\item   $\ms_{\gm}\neq\emptyset$ for any simple finite-dimensional Lie algebra $\lie g$.
\item For any simply-laced simple Lie algebra, there exists $\underline{w}_0$ such that $\mathcal{S}_{\gm}\cap\mathcal{D}_{\underline{w}_0}^q$ is non-empty.
\end{enumerate}
\end{conjecture*}

\subsection{Remarks on the boundary of the classical degree cone}
Let $\g$ be of type $\tt A_n$. The boundary of $\mathcal{D}$, denoted by $\partial\mathcal{D}$, is defined as the difference of the closure of $\mathcal{D}$ and its relative interior. Let $S(\partial\md):=\partial\md\cap\mathbb{Z}^{\Delta_+}$ be the lattice points in $\partial\md$.
\par
Let $\bold{d}\in S(\partial\md)$. Then $\bd$ defines a filtration $\ff^\bd$ on $U(\lie n^-)$. In general, the associated graded algebra is no longer the commutative algebra $S(\lie n^-)$, but some algebra which is a degenerated version of $U(\lie n^-)$ and admits a quotient $S(\lie n^-)$. 
This associated graded algebra is the universal enveloping algebra of the Lie algebra $\lie n^{-,\bd}$, which is a contraction the Lie bracket of $\lie n^-$ on the prescribed roots by $\bd$ (see \cite{CIFFFR} for details).
\par
For $\lambda\in\mathcal{P}_+$, we can similarly define the associated graded module $V^{\bd}(\lambda)$: it is a cyclic $U(\lie n^{-,\bd})$-module with cyclic vector $v_\lambda^\bd$. It is proved in \cite{CIFFFR} that the highest weight orbit 
$$\ff^\bd(\lambda):=\overline{\exp(\lie n^{-,\bd})\cdot [v_\lambda^\bd]}\subset \mathbb{P}(V^\bd(\lambda))$$
is a flat degeneration of the partial flag variety $\ff(\lambda)$. 
\par
Moreover, for some $\bold{d}\in S(\partial\md)$, it is conjectured in \cite{FaFo} that a monomial basis of the representation $V^\bd(\lambda)$ can be parametrized by the lattice points in a chain-order polytope associated to a marked poset.

\subsection{Organization of paper}
In Section~\ref{sec2}, we fix notations, introduce the classical degree cone, the local and global monomial sets. Quantum degree cones are defined in Section~\ref{sec3}, where Theorem A is proved. We provide examples and properties of the quantum degree cone in Section~\ref{sec4}. In Section~\ref{sec5}, examples for local and global monomials sets are given and Theorem B is proved. We conclude with some examples on quantum degree cones in Section~\ref{sec6}.

\subsection{Acknowledgements}
X.F. is supported by the Alexander von Humboldt Foundation. T.B. is partially supported by the DFG priority program SPP 1388 "Representation Theory".

\section{Lie algebras and the classical degree cones}\label{sec2}
\subsection{Notations and basic properties}\label{Sec:Notation}
Let $\g$ be a simple Lie algebra of rank $n$ over the field of complex numbers $\mc$. We fix a triangular decomposition $\g=\mathfrak{n}^+\oplus\mathfrak{h}\oplus\mathfrak{n}^-$ and a set of simple roots $\Pi=\{\alpha_1,\ldots,\alpha_n\}$ of $\g$. The set of positive roots of $\g$ will be denoted by $\Delta_+$ with cardinality $N$. Let $\mathcal{Q}_+=\sum_{i=1}^n \mathbb{N}\varpi_i$ be the root monoid. Let $\rho=\frac{1}{2}\sum_{\alpha\in\Delta_+}\alpha$ be the half sum of positive roots. For $\alpha\in\Delta_+$, we pick a root vector $f_\alpha$ of weight $-\alpha$. Let $\varpi_i$, $i=1,\ldots,n$ be the fundamental weights, $\mathcal{P}$ be the weight lattice and $\mathcal{P}_+=\sum_{i=1}^n\mathbb{N}\varpi_i$ be the set of dominant integral weights. For a dominant integral weight $\lambda\in\mathcal{P}_+$, let $V(\lambda)$ be the finite dimensional irreducible representation of $\g$ of highest weight $\lambda$ and $v_\lambda$ a highest weight vector. Let $U(\lie n^-)$ be the enveloping algebra of $\lie n^-$ and $S(\lie n^-)$ be the symmetric algebra of $\lie n^-$. For $\lambda=\sum_{i=1}^n\lambda_i\varpi_i\in\mathcal{P}_+$, denote the height of $\lambda$ by: $|\lambda|:=\sum_{i=1}^n m_i$.
\par
We define $\mathbb{R}^{\Delta_+}:=\{f:\Delta_+\ra\mathbb{R}\text{ is a function}\}$. It is an $\mathbb{R}$-vector space of dimension $N$. Let $\mathbb{R}^{\Delta_+}_{\geq 0}\subset \mathbb{R}^{\Delta_+}$ be the set of functions taking positive values, we define similarly $\mathbb{N}^{\Delta_+}$ and $\mathbb{Z}^{\Delta_+}$. A function $\bold{d}\in\mathbb{R}^{\Delta_+}$ is determined by its values $(d_{\beta}:=\bold{d}(\beta))_{\beta\in\Delta_+}$. Once a sequence of positive roots $(\beta_1,\beta_2,\ldots,\beta_N)$ is fixed, $\mathbb{R}^{\Delta_+}$ is identified with $\mathbb{R}^N$ via: $\bold{d}\mapsto (d_{\beta_1},d_{\beta_2},\ldots,d_{\beta_N})$.
\par
Let $W$ be the Weyl group of $\g$ with generators $s_1,\ldots,s_n$ and $w_0\in W$ be the longest element. We denote $R(w_0)$ the set of all reduced decompositions of $w_0$. 
\par
For any reduced decomposition $\underline{w}_0=s_{i_1}\ldots s_{i_N}\in R(w_0)$, we associate a convex total order on $\Delta_+$: for $1\leq t\leq N$, we denote $\beta_t=s_{i_1}\ldots s_{i_{t-1}}(\alpha_{i_t})$, then $\Delta_+=\{\beta_1,\ldots,\beta_N\}$ and $\beta_1<\beta_2<\ldots<\beta_N$ is the desired convex total order. It is proved by Papi \cite{P94} that the above association induces a bijection between $R(w_0)$ and the set of all convex total orders on $\Delta_+$.
\par
For the simple Lie algebra $\lie {sl}_{n+1}$ of type $\tt A_n$ and $1\leq i\leq j\leq n$, we denote $\alpha_{i,j}:=\alpha_i+\ldots+\alpha_j$, then $\Delta_+=\{\alpha_{i,j}\mid 1\leq i\leq j\leq n\}$. For the simple Lie algebra of type $\tt B_n$ and $1\leq i\leq j\leq n$, we denote $\alpha_{i,j}:=\alpha_i+\ldots+\alpha_j$ and for $\alpha_{i,\overline{j}}:=\alpha_i+\ldots+\alpha_n+\alpha_n+\ldots+\alpha_j$, then $\Delta_+=\{\alpha_{i,j},\alpha_{k,\overline{l}}\mid 1\leq i\leq j\leq n,\ 1\leq k< l\leq n\}$. For the simple Lie algebra $\lie {sp}_{2n}$ of type $\tt C_n$ and $1\leq i\leq j\leq n$, we denote $\alpha_{i,j}:=\alpha_i+\ldots+\alpha_j$ and $\alpha_{i,\overline{j}}:=\alpha_i+\ldots+\alpha_n+\ldots+\alpha_j$, notice that $\alpha_{i,n}=\alpha_{i,\overline{n}}$, then $\Delta_+=\{\alpha_{i,j},\alpha_{i,\overline{j}}\mid 1\leq i\leq j\leq n\}$.
\par
For $\bold{s}=(s_\alpha)_{\alpha\in\Delta_+}\in\mathbb{N}^{\Delta_+}$, we denote $f^{\bold{s}}:=\prod_{\alpha\in\Delta_+}f_\alpha^{s_\alpha}\in S(\lie n^-)$. For any $\bold{d}\in\mathbb{N}^{\Delta_+}$, we denote $\textrm{deg}_{\bold{d}}(f^\bold{s}):=\sum_{\alpha\in\Delta_+} s_\alpha d_\alpha$.

\subsection{The classical degree cone}

We start with the classical degree cone.

\begin{definition}
The classical degree cone $\md$ is defined by:
\[
\mathcal{D}:=\{\bold{d}\in\mathbb{R}_{\geq 0}^{\Delta_+}\mid \text{for any $\alpha,\beta,\gamma\in\Delta_+$ such that $\alpha+\beta=\gamma$, }d_\alpha+d_\beta>d_\gamma\}.
\]
\end{definition}
\begin{example} The element $\mathbf{e}$ defined by $e_\alpha = 1$ for all $\alpha$ is in $\mathcal{D}$ for any simple Lie algebra.
\end{example}
By definition, $\md$ is an open polyhedral cone. We let $S(\md):=\md\cap\mathbb{Z}^{\Delta_+}$ denote the set of lattice points in $\md$. For any $\bold{d}=(d_{\beta})_{\beta\in\Delta_+}\in S(\md)$, we define a filtration $\ff^{\,\bold{d}}$ on $U(\lie n^-)$ by:
$$\ff^{\,\bold{d}}_{s}U(\lie n^-):=\sspan\{f_{\gamma_1}f_{\gamma_2}\ldots f_{\gamma_k}\ |\ \gamma_1,\ldots,\gamma_k\in\Delta_+\ \text{such that}\  d_{\gamma_1}+d_{\gamma_2}+\ldots+d_{\gamma_k}\leq s\}.$$
By the cyclicity, every irreducible representation $V(\lambda)$ admits a filtration arising from $\ff^{\,\bold{d}}$:
$$\ff^{\,\bold{d}}_{s}V(\lambda):=\ff^{\,\bold{d}}_{s}U(\lie n^-)\cdot v_\lambda.$$
Note that for $\mathbf{d} = \mathbf{e}$, this is the PBW filtration, which has been subject to a lot of researches in the past ten years.

\begin{lemma}\it
For any $\bold{d}\in S(\md)$, we have:
\begin{enumerate}
\item $\ff^{\,\bold{d}}:=(\ff^{\,\bold{d}}_{0}\subset \ff^{\,\bold{d}}_{1}\subset\ldots\subset \ff^{\,\bold{d}}_{n}\subset\ldots)$ defines a filtration on $U(\lie n^-)$ whose associated graded algebra is isomorphic to $S(\lie n^-)$.
\item Let $V^\bold{d}(\lambda)$ be the graded module associated to the induced filtration. Then $V^\bold{d}(\lambda)$ is a cyclic $S(\lie n^-)$-module.
\end{enumerate}
\end{lemma}

\begin{proof}
The universal enveloping algebra $U(\lie n^-)$ is a quotient of the tensor algebra $T(\lie n^-)$ by the ideal generated by $x\ts y-y\ts x-[x,y]$ for all $x,y\in \lie n^-$. In $\lie n^-$, for $\alpha,\beta,\gamma\in\Delta_+$ with $\alpha+\beta=\gamma$, $[f_\alpha,f_\beta]$ is a multiple of $f_\gamma$; if $\bold{d}\in\mathcal{D}$, we have $d_\alpha+d_\beta>d_\gamma$, which proves the first part of the lemma. The second part is clear.
\end{proof}

Let $v_\lambda^\bold{d}$ be a cyclic vector in $V^\bold{d}(\lambda)$. By (2) of the lemma, the $S(\lie n^-)$-module map 
$$\vp:S(\lie n^-)\ra V^\bold{d}(\lambda),\ \  x\mapsto x\cdot v_\lambda^\bold{d}$$
is surjective. We denote $I^\bold{d}(\lambda):=\ker\vp$ and call it the defining ideal of $V^\bold{d}(\lambda)$.

\subsection{The local and global monomial set}
We are interested in some particular degrees such that the associated graded module admits "good" bases.

\begin{definition}
The \emph{local monomial set} $\mathcal{S}_{\lm}$ is defined by:
\[
\mathcal{S}_{\lm}:=\{\bold{d}=(d_\beta)_{\beta\in\Delta_+}\in S(\md)\ |\ \text{for any}\ i=1,2,\ldots,n,\ I^\bold{d}(\varpi_i)\ \text{is a monomial ideal}\}.
\]
\end{definition}
\begin{remark} 
For any simple Lie algebra $\g$, the local monomial set $\mathcal{S}_{\lm}$ is non-empty. For example, one possibility is to linearly order a monomial basis of any fixed regular representation. The induced order will be in the local monomial set.
\end{remark}
\begin{definition}
The \emph{global monomial set} $\mathcal{S}_{\gm}$ is defined by:
\[
\mathcal{S}_{\gm}:=\{\bold{d}=(d_\beta)_{\beta\in\Delta_+}\in S(\md)\ |\ \text{for any}\ \lambda\in\mathcal{P}_+,\ I^\bold{d}(\lambda)\ \text{is a monomial ideal}\}.
\]
\end{definition}

It is clear that $\mathcal{S}_{\gm}\subset \mathcal{S}_{\lm}$. \par
The main goal of this paper is to study the following questions:
\begin{enumerate}
\item Whether the global monomial set $\mathcal{S}_{\gm}$ is non-empty? That is to say, does there exist a filtration on $U(\lie n^-)$ arising from $\bold{d}\in\md$ such that for any finite dimensional irreducible representation, the defining ideal of the associated graded module is monomial?
\item If the answer to the above question is affirmative, for any $\lambda\in\mathcal{P}_+$, we obtain a unique monomial basis for $V^{\bold{d}}(\lambda)$ parametrized by $S(\lambda):=\{\bold{s}\in\mathbb{N}^{\Delta_+}\mid f^{\bold{s}}\cdot v_\lambda^{\bold{d}}\neq 0\}$. Whether there exists a lattice polytope $P(\lambda)$ such that $S(\lambda)$ is exactly the lattice points in $P(\lambda)$?
\end{enumerate}

\subsection{Criteria for the local monomial set}
We first give a criterion to decide whether $\bold{d}\in\mathcal{D}$ is contained in $\mathcal{S}_{\lm}$, which is useful in the rest of the paper.
\par
We fix $\lambda\in\mathcal{P}_+$. For $\mu\in\mathcal{P}$ such that $r_\mu:=\dim V(\lambda)_\mu\neq 0$, we denote 
$$S_\mu:=\{\bold{s}\in\mathbb{N}^{\Delta_+}\mid f^{\bold{s}}\cdot v_\lambda\neq 0\in V(\lambda)\ \ \text{and}\ \ \sum_{\alpha\in\Delta_+}s_\alpha\alpha=\lambda-\mu\}:$$
this is a finite set. Suppose that $S_\mu=\{\bold{s}_1,\bold{s}_2,\ldots,\bold{s}_{m_\mu}\}$ with
$$\deg_{\bd}(f^{\bold{s}_1})\leq \deg_{\bd}(f^{\bold{s}_2})\leq\ldots\leq \deg_{\bd}(f^{\bold{s}_{m_\mu}}).$$

Let $T_\mu=\{\bold{s}_k\mid f^{\bold{s}_k}\cdot v_\lambda\notin\text{span}\{f^{\bold{s}_1}\cdot v_\lambda,\ldots,f^{\bold{s}_{k-1}}\cdot v_\lambda\}\}$.
Then by construction the set $\{ f^{\mathbf{s}}.v_\lambda \mid \mathbf{s} \in T_{\mu} \}$ is a basis.
\begin{lemma}\label{Lem:useful}
Let  $\mu\in\mathcal{P}$. Suppose
\[
\mathbf{s}_{k} \notin T_{\mu} \Rightarrow \deg(f^{\mathbf{s}_k}) >  \deg(f^{\mathbf{s}_l}) \text{ for all }  \mathbf{s}_l \in T_{\mu} \text{ with } l < k,\]
then the defining ideal $I^\bold{d}(\lambda)$ is monomial.
\end{lemma}

\begin{proof}
It suffices to show that if $\bold{s}_k\notin T_\mu$, then $f^{\bold{s}_k}\in I^\bold{d}(\lambda)_\mu$. Indeed, by definition, $\bold{s}_k\notin T_\mu$ implies that $f^{\bold{s}_k}\cdot v_\lambda$ is a linear combination of $f^{\bold{s}_{i_1}}\cdot v_\lambda$, $\ldots$, $f^{\bold{s}_{i_p}}\cdot v_\lambda$ with $\mathbf{s}_{i_q} \in T_{\mu}$. By assumption, $\deg(f^{\mathbf{s}_k}) >  \deg(f^{\mathbf{s}_{i_q}})$, hence in the graded module we have $f^{\bold{s}_k}\in I^{\bold{d}}(\lambda)_\mu$.
\end{proof}

The following corollary is a special case of the lemma; it will be used repeatedly when dealing with the examples.

\begin{corollary}\label{Cor:useful}
The defining ideal $I^\bold{d}(\lambda)$ is monomial, if for any $\mu\in\mathcal{P}$ with $r_\mu\neq 0$:
\begin{enumerate}
\item if $r_\mu=1$, $\deg_{\bd}(f^{\bs_1})<\deg_{\bd}(f^{\bs_2})$;
\item if $r_\mu>1$, $\#\{\deg_{\bd}(f^{\bs_k})\mid 1\leq k\leq m_\mu\}=m_\mu$.
\end{enumerate}
\end{corollary}

\subsection{How local and global monomial sets are related}

We give a sufficient condition for an element in $\mathcal{S}_{\lm}$ being contained in $\mathcal{S}_{\gm}$.  Let $\bold{d}\in\ms_{\lm}$ and for $1\leq i\leq n$, $S^{\bold{d}}(\varpi_i)=\{\bold{a}\in\mathbb{N}^{\Delta_+}\mid f^\bold{a}\cdot v_{\varpi_i}^\bold{d}\neq 0\}$. For an integer $m\geq 1$, let $S^{\bold{d}}(\varpi_i)^{+m}$ denote the $m$-fold Minkowski sum of $S^{\bold{d}}(\varpi_i)$. We will write $S(\varpi_i)$ to instead $S^{\bold{d}}(\varpi_i)$ when the context is clear.

\begin{theorem}\label{Thm:Minkow}
\it For any $\lambda=m_1\varpi_1+m_2\varpi_2+\ldots+m_n\varpi_n\in\mathcal{P}_+$, if $\#(S(\varpi_1)^{+m_1}+S(\varpi_2)^{+m_2}+\ldots+S(\varpi_n)^{+m_n})=\dim V(\lambda)$, then $\bold{d}\in\ms_{\gm}$.
\end{theorem}
The rest of this paragraph is devoted to the proof of this statement. It is in adaptation of the proof in \cite[Theorem 3]{FFR15}. For the convenience of the reader, we provide some key points of this proof, details can be found in [\textit{loc.cit}].
For any $\tau= \sum_{i=1}^n r_i\varpi_i \in \mathcal{P}_+$, we define 
$$S(\tau) := S(\varpi_1)^{+r_1} + S(\varpi_2)^{+r_2} + \ldots + S(\varpi_n)^{+r_n}.$$
We want to show simultaneously that: for $\lambda,\mu \in \mathcal{P}_+$,
\begin{itemize}
\item[(1)] $\{ f^{\bs}\cdot v^{\bd}_{\lambda + \mu} \mid \bs \in S(\lambda + \mu) \}$ is a basis of $V^{\bd}(\lambda + \mu)$;
\item[(2)] the defining ideal $I^{\bd}(\lambda + \mu)$ is monomial.
\end{itemize}
\par
The statements will be proved by induction on the height of $\lambda+\mu$. The height $1$ case is the assumption $\bd \in \ms_{\lm}$. The induction step will be divided into several parts.
\par
Let $<$ be a total order on $\{f_\beta \mid \beta \in \Delta_+\}$ refining the partial order defined by $\bd = (d_{\beta})_{\beta \in \Delta_+}$ and consider the induced lexicographical order on the monomials in $U(\n^-)$. The following proposition is proved essentially in \cite[Proposition 2.11]{FFL3}; in \cite[Proposition 4]{FFR15}, it is proved in detail for a particular degree function for type $\tt A_n$, but the proof used there is valid for a general $\bd$.

\begin{proposition}
For any $\lambda, \mu \in \mathcal{P}^+$ the set $\{ f^{\bs}\cdot (v^{\bd}_\lambda \otimes v^{\bd}_\mu) \mid \bs \in S(\lambda + \mu)\}$ is linear independent in $V^{\bd}(\lambda) \otimes V^{\bd}(\mu)$.
\end{proposition} 

This set lies in the Cartan component of $V^{\bd}(\lambda) \otimes V^{\bd}(\mu)$ and since $|S(\lambda+\mu)| = \dim V(\lambda + \mu) $ this set is a basis of the Cartan component of $V(\lambda) \otimes V(\mu)$ and of $V(\lambda + \mu)$ respectively.
\begin{proposition}\label{Prop:eqzero}
If $\bs \notin S(\lambda + \mu)$, then $f^{\bs}\cdot v^{\bd}_{\lambda+\mu} = 0$ in $V^{\bd}(\lambda + \mu)$.
\end{proposition} 
\begin{proof}

We fix $\bs \notin S(\lambda + \mu)$ and write
\begin{equation}\label{expansion}
f^{\bs}\cdot v_{\lambda + \mu} = \sum_{\bt \in S(\lambda+\mu)} c_{\bt} f^{\bt}\cdot v_{\lambda + \mu} \ \mbox{ in } V(\lambda + \mu).
\end{equation}

Since $V(\lambda + \mu) \subset V(\lambda) \otimes V(\mu)$ we have an expansion of the Equation \eqref{expansion}:
$$
f^{\bs}\cdot (v_\lambda \otimes v_{\mu}) = \!\!\!\!\!\! \sum_{\bt \in S(\lambda + \mu), \ \mathbf{a} + \mathbf{b} = \bt}\!\!\!\!\!\! c_\bt c_{\mathbf{a},\mathbf{b}} f^{\mathbf{a}}\cdot v_\lambda \otimes f^{\mathbf{b}}\cdot v_\mu \ \mbox{ in } V(\lambda) \otimes V(\mu).
$$
By replacing those $\mathbf{a} \notin S(\lambda)$ (resp. $\mathbf{b} \notin S(\mu)$) by a sum supported on $S(\lambda)$ (resp. $S(\mu)$) we obtain a unique expression. By induction, the corresponding monomials have strictly lower degrees then $\deg(f^{\mathbf{a}})$ (resp. $\deg(f^{\mathbf{b}})$). 
This implies that we have for all $\bt$ appearing in this unique expression $\deg(f^{\bt}) < \deg(f^{\bs})$.
\end{proof}
\begin{proposition}
The set $\mathcal{B} = \{f^{\bs}\cdot v_{\lambda+\mu}^{\bd} \mid \bs \in S(\lambda + \mu)\}$ is a basis of $V^{\bd}(\lambda + \mu)$.
\end{proposition} 
\begin{proof}
By considering each filtration component, this is a direct consequence of Proposition~\ref{Prop:eqzero}.
\end{proof}

We are left with proving (2), the monomiality of the annihilating ideal. This follows immediately from the following lemma.
\begin{proposition}
The defining ideal of the Cartan component of $V^{\bd}(\lambda) \otimes V^{\bd}(\mu)$ is monomial and there exists an $S(\n^-)$-module isomorphism from the Cartan component of $V^{\bd}(\lambda) \otimes V^{\bd}(\mu)$ to $V^{\bd}(\lambda + \mu)$.
\end{proposition} 
\begin{proof}
We have for $\bs \notin S(\lambda + \mu) = S(\lambda) + S(\mu)$:
\begin{equation*}\label{expansion2}
f^{\bs}\cdot (v_\lambda \otimes v_{\mu}) = \sum_{\bt_1 + \bt_2 = \bs} f^{\bt_1}\cdot v_\lambda \otimes f^{\bt_2}\cdot v_{\mu} \mbox{ in } V(\lambda) \otimes V(\mu)
\end{equation*}
and $\bt_1 \notin S(\lambda)$ or $\bt_2 \notin S(\mu)$. Hence by Proposition \ref{Prop:eqzero} we can conclude that either $f^{\bt_1}\cdot v_\lambda^{\bd} = 0$ in $V^{\bd}(\lambda)$ or $f^{\bt_2}\cdot v_\mu^{\bd} = 0$ in $V^{\bd}(\mu)$. We obtain
\begin{equation}\label{Eq:monomiality}
f^{\bs}\cdot (v_\lambda^{\bd} \otimes v_\mu^{\bd}) = 0 \ \mbox{ in } V^{\bd}(\lambda) \otimes V^{\bd}(\mu).
\end{equation}
Therefore we have monomiality. 
\par
By Proposition \ref{Prop:eqzero}, there is a surjective map of $S(\n^-)$-modules from the Cartan component of $V^{\bd}(\lambda) \otimes V^{\bd}(\mu)$ to $V^{\bd}(\lambda + \mu)$, which is an isomorphism for dimension reasons.\end{proof} 
This proves the monomiality statement (2) and hence Theorem \ref{Thm:Minkow}, i.e. $\bd \in \mathcal{S}_{\gm}$.

\section{Quantum groups and quantum degree cones}\label{sec3}
\subsection{Quantum groups}
Let $C=(c_{ij})_{n\times n}\in\Mat_n(\mathbb{Z})$ be the Cartan matrix of $\g$ and $D=\diag(d_1,\ldots,d_n)\in\Mat_n(\mathbb{Z})$ be a diagonal matrix symmetrizing $C$. Thus $A=DC=(a_{ij})_{n\times n}\in\Mat_n(\mathbb{Z})$ is symmetric. Let $U_q(\g)$ be the corresponding quantum group over $\mc(q)$: as an algebra, it is generated by $E_i$, $F_i$ and $K_i^{\pm 1}$ for $i=1,\ldots,n$, subject to the following relations: for $i,j=1,\ldots,n$,
$$K_iK_i^{-1}=K_i^{-1}K_i=1,\ \ K_iE_jK_i^{-1}=q_i^{c_{ij}}E_j,\ \ K_iF_jK_i^{-1}=q_i^{-c_{ij}}F_j,$$
$$E_iF_j-F_jE_i=\delta_{ij}\frac{K_i-K_i^{-1}}{q_i-q_i^{-1}},$$
and for $i\neq j$,
$$\sum_{r=0}^{1-c_{ij}}(-1)^r E_i^{(1-c_{ij}-r)}E_jE_i^{(r)}=0,\ \ \sum_{r=0}^{1-c_{ij}}(-1)^rF_i^{(1-c_{ij}-r)}F_jF_i^{(r)}=0,$$
where 
\[
q_i=q^{d_i},\, \, [n]_q!=\prod_{i=1}^n \frac{q^n-q^{-n}}{q-q^{-1}},\  \ E_i^{(n)}=\frac{E_i^n}{[n]_{q_i}!}\ \ \text{and}\ \ F_i^{(n)}=\frac{F_i^n}{[n]_{q_i}!}.
\]

Let $U_q(\mathfrak{n}^-)$ be the sub-algebra of $U_q(\g)$ generated by $F_i$ for $i=1,\ldots,n$. For $\lambda\in\mathcal{P}_+$, we denote by $V_q(\lambda)$ the finite dimensional irreducible representation of $U_q(\g)$ of highest weight $\lambda$ and type $1$ with highest weight vector $\mathbf{v}_\lambda$.
\par
When $q$ is specialized to $1$, the quantum group $U_q(\g)$ admits $U(\g)$ as its classical limit. In this limit, the representation $V_q(\lambda)$ is specialized to $V(\lambda)$.

\subsection{PBW root vectors and commutation relations}
Let $T_i=T_{i,1}''$, $i=1,\ldots,n$ be Lusztig's automorphisms:
$$T_i(E_i)=-F_iK_i,\ \ T_i(F_i)=-K_i^{-1}E_i,\ \ T_i(K_j)=K_jK_i^{-c_{ij}},$$
for $i=1,\ldots,n$, and $j\neq i$,
$$T_i(E_j)=\sum_{r+s=-c_{ij}}(-1)^rq_i^{-r}E_i^{(s)}E_jE_i^{(r)},\ \ T_i(F_j)=\sum_{r+s=-c_{ij}}(-1)^rq_i^{r}F_i^{(r)}F_jF_i^{(s)}.$$
For details, see Chapter 37 in \cite{Lus2}. We fix a reduced decomposition $\underline{w}_0=s_{i_1}\ldots s_{i_N}\in R(w_0)$ and let positive roots $\beta_1,\beta_2,\ldots,\beta_N$ be defined as in Section \ref{Sec:Notation}. The quantum PBW root vector $F_{\beta_t}$ associated to a positive root $\beta_t$ is defined by:
$$F_{\beta_t}=T_{i_1}T_{i_2}\ldots T_{i_{t-1}}(F_{i_t})\in U_q(\n^-).$$
The PBW theorem of quantum groups affirms that the set
$$\{F^\bold{s}:=F_{\beta_1}^{s_1}F_{\beta_2}^{s_2}\ldots F_{\beta_N}^{s_N}|\ \bold{s}=(s_1,\ldots,s_N)\in\mathbb{N}^N\}$$
forms a $\mc(q)$-basis of $U_q(\mathfrak{n}^-)$ (\cite{Lus2}, Corollary 40.2.2).
\par
The commutation relations between these quantum PBW root vectors are given by the following Levendorskii-Soibelman (L-S for short) formula:
for any $i<j$, 
\begin{equation}\label{Eq:LS}
F_{\beta_j}F_{\beta_i}-q^{-(\beta_i,\beta_j)}F_{\beta_i}F_{\beta_j}=\sum_{n_{i+1},\ldots,n_{j-1}\geq 0}c(n_{i+1},\ldots,n_{j-1})F_{\beta_{i+1}}^{n_{i+1}}\ldots F_{\beta_{j-1}}^{n_{j-1}},
\end{equation}
where $c(n_{i+1},\ldots,n_{j-1})\in\mathbb{C}[q^{\pm 1}]$. We denote 
$$M_{i,j} = \{F^{n_{i+1}}_{\beta_{i+1}}F^{n_{i+2}}_{\beta_{i+2}} \ldots F^{n_{j-1}}_{\beta_{j-1}} \mid n_{i+1}\beta_{i+1} + n_{i+2} \beta_2 + \dots + n_{j-1}\beta_{j-1} = \beta_i+\beta_j \},$$
then for weight reasons, the sum in the right hand side of \eqref{Eq:LS} is supported inside $M_{i,j}$. Denote by $M^q_{i,j}\subset M_{i,j}$ the set of monomials which actually appear with a non-zero coefficient in the right-hand side of \eqref{Eq:LS}. It should be pointed out that the right hand side of \eqref{Eq:LS} largely depends on the chosen reduced decomposition. In general it is hard to know which monomials appear in $M_{i,j}^q$.
\par
Let us have a closer look on how these formulas depend on the reduced decomposition. Let $\underline{w}_0$, $\underline{w}_0'\in R(w_0)$ be two reduced decompositions such that they are of form
$$\underline{w}_0=\underline{w}_Ls_ps_q\underline{w}_R,\ \ \underline{w}_0'=\underline{w}_Ls_qs_p\underline{w}_R$$
with $1\leq p\neq q\leq n$ and $s_ps_q=s_qs_p$. We define $l=\ell(\underline{w}_L)$. 
\par
Let the convex total order on $\Delta_+$ induced by $\underline{w}_0$ (resp. $\underline{w}_0'$) be:
$$\beta_1<\beta_2<\ldots<\beta_N\ \ (\text{resp.}\ \beta_1'<\beta_2'<\ldots<\beta_N').$$
For $s\leq l$, the L-S formula (\ref{Eq:LS}) reads:
\begin{equation}\label{Eq:LS1}
F_{\beta_{s}}F_{\beta_{l+2}}-q^{(\beta_s,\beta_{l+2})}F_{\beta_{l+2}}F_{\beta_{s}}=\sum_{n_{s+1},\ldots,n_{l+1}\geq 0}c(n_{s+1},\ldots,n_{l+1})F_{\beta_{s+1}}^{n_{s+1}}\ldots F_{\beta_{l+1}}^{n_{l+1}}.
\end{equation}
For $t\geq l+3$, the L-S formula (\ref{Eq:LS}) reads:
\begin{equation}\label{Eq:LS2}
F_{\beta_{t}}F_{\beta_{l+1}}-q^{-(\beta_t,\beta_{l+1})}F_{\beta_{l+1}}F_{\beta_{t}}=\sum_{n_{l+2},\ldots,n_{t-1}\geq 0}c(n_{l+2},\ldots,n_{t-1})F_{\beta_{l+2}}^{n_{l+2}}\ldots F_{\beta_{t-1}}^{n_{t-1}}.
\end{equation}

\begin{lemma}\label{Lem:samecone}
\it In the formula \eqref{Eq:LS1}, $n_{l+1}=0$; in the formula \eqref{Eq:LS2}, $n_{l+2}=0$.
\end{lemma}

\begin{proof}
We prove for example the first statement, the second one can be shown similarly.
\par
First notice that for any $i\neq l+1$, $l+2$, $\beta_i=\beta_i'$, $\beta_{l+1}=\beta_{l+2}'$, $\beta_{l+2}=\beta_{l+1}'$. The same argument can be applied to quantum PBW root vectors: let $F_{\beta_1}, F_{\beta_2},\ldots,F_{\beta_N}$ (resp. $F_{\beta_1}', F_{\beta_2}',\ldots,F_{\beta_N}'$) be the quantum PBW root vectors obtained from $\underline{w}_0$
(resp. $\underline{w}_0'$). Then for any $i\neq l+1,l+2$, $F_{\beta_i}=F_{\beta_i}'$, $F_{\beta_{l+1}}=F_{\beta_{l+2}}'$, $F_{\beta_{l+2}}=F_{\beta_{l+1}}'$.
For $s\leq l$, we apply the L-S formula to $F_{\beta_{s}}'$ and $F_{\beta_{l+1}}'$, it gives:
$$F_{\beta_{s}}'F_{\beta_{l+1}}'-q^{(\beta_s',\beta_{l+1}')}F_{\beta_{l+1}}'F_{\beta_{s}}'=\sum_{m_{s+1},\ldots,m_{l}\geq 0}d(m_{s+1},\ldots,m_{l}) F_{\beta_{s+1}}'^{m_{s+1}}\ldots F_{\beta_{l}}'^{m_{l}}.$$
Comparing it to \eqref{Eq:LS1} gives $n_{l+1}=0$.
\end{proof}

\subsection{Quantum degree cones}
We fix in this paragraph a reduced decomposition $\underline{w}_0\in R(w_0)$ and positive roots $\beta_1,\ldots,\beta_N$ obtained from $\underline{w}_0$ as explained in Section \ref{Sec:Notation}. 

\begin{definition}
The \emph{quantum degree cone} associated to $\underline{w}_0$ is defined by:
$$\small \md^q_{\underline{w}_0}:=\{(d_{\beta})_{\beta\in\Delta_+}\in\mathbb{R}_{\geq 0}^{\Delta_+}\ |\ \text{$\forall i<j$, $d_{\beta_i}+d_{\beta_j}>\sum_{k=i+1}^{j-1}n_kd_{\beta_k}$ if $c(n_{i+1},\ldots,n_{j-1})\neq 0$ in \eqref{Eq:LS}}\}.$$
We denote the set
$$\md^q:=\bigcup_{\underline{w}_0\in R(w_0)}\md^q_{\underline{w}_0}\subset\mathbb{R}_{\geq 0}^{\Delta_+}.$$
\end{definition}

Let $\md\subset\mathbb{R}_{\geq 0}^{\Delta_+}$ be the classical degree cone. Specializing the quantum parameter $q$ to $1$ proves the following lemma:

\begin{lemma}\label{Lem:inclusion}
We have $\md^q\subset\md$.
\end{lemma}

\begin{remark}
Except for small rank cases $\g=\lie{sl}_2$, $\lie{sl}_3$ (see Example \ref{Ss:QDC:A2}), the inclusion in Lemma \ref{Lem:inclusion} is strict. For example, the constant function $\bold{1}$ taking value $1$ on each positive root
is in the classical degree cone $\mathcal{D}$, but for $\g\neq\lie{sl}_2, \lie{sl}_3$, there is no reduced decomposition $\underline{w}_0$ such that $\bold{1}\notin \Dwq$. See for example \cite[Section 2.4]{FFR15} and Example \ref{Ss:QDC:C2} for type $\tt C_2$, Section \ref{Sec:G2} for type $\tt G_2$.
\end{remark}

Let $\bold{d}=(d_{\beta})_{\beta\in\Delta_+} \in  S(\Dwq) =:\Dwq  \cap \N^{\Delta_+}$. For a monomial $F^\bold{t}$ where $\bold{t}=(t_1,\ldots,t_N)$, we define its $\bold{d}$-degree $\ddeg_{\bold{d}}$ by:
$${\ddeg}_{\bold{d}}(F^\bold{t}):=t_1d_{\beta_1} + t_2d_{\beta_2}+\ldots + t_Nd_{\beta_N}.$$

\begin{theorem}\label{thm:polycone}
The set $\md_{\underline{w}_0}^q$ is a non-empty open polyhedral cone. \end{theorem}
\begin{proof}
By definition, $\md^q_{\underline{w}_0}$ is an open polyhedral cone. We describe an inductive procedure to construct an element $\bold{d}=(d_{\beta_1},\ldots,d_{\beta_N})\in\md^q_{\underline{w}_0}$.
\par
We set $d_{\beta_1}=1$. Suppose that $d_{\beta_1},\ldots,d_{\beta_k}$ are chosen such that they satisfy the inequalities in the definition of $\md_{\underline{w}_0}^q$.
\par
Let $M_{k+1}^q:=\bigcup_{s=1}^k M_{s,k+1}^q$. Since $M_{k+1}^q$ is a finite set, we set 
$$d_{\beta_{k+1}}=1+\max_{F^\bold{t}\in M_{k+1}^q}(\deg_{\bold{d}}(F^\bold{t})).$$
Since $F^\bold{t}\in M_{k+1}^q$ is a monomial on $\{F_{\beta_1},\ldots,F_{\beta_k}\}$, the degree is well-defined. By definition, for any $1\leq s\leq k$ and any $F^\bold{t}\in M_{s,k+1}^q$, $d_{\beta_s}+d_{\beta_{k+1}}>\deg_{\bold{d}}(F^\bold{t})$. This terminates the proof.
\end{proof}

For $\bold{d}\in S(\md_{\underline{w}_0}^q)$, we define a filtration $\ff^\bold{d}_\bullet =(\ff_0^\bold{d}\subset \ff_1^\bold{d}\subset\ldots\subset \ff_n^\bold{d}\subset\ldots)$ on $U_q(\lie n^-)$  by:
$$\ff_k^\bold{d}U_q(\mathfrak{n}^-):=\sspan\{F^\bold{t}\ |\ {\ddeg}_{\bold{d}}(F^\bold{t}) \leq k\}.$$
Let $S_q(\lie n^-)$ be the algebra generated by $x_1,x_2,\ldots,x_N$, subjects to the following relations: for $1\leq i<j\leq N$,
$$x_ix_j=q^{(\beta_i,\beta_j)}x_jx_i.$$
The following proposition is clear from the L-S formula (\ref{Eq:LS}).
\begin{proposition}\label{Prop:qFiltgeneral}
\begin{enumerate}
\item The filtration $\ff_\bullet$ endows $U_q(\mathfrak{n}^-)$ with a filtered algebra structure.
\item The associated graded algebra ${\gr}_{\ff}U_q(\lie n^-)$ is  isomorphic to $S_q(\mathfrak{n}^-)$.
\end{enumerate}
\end{proposition}

For $\lambda\in\mathcal{P}_+$, the above filtration on $U_q(\lie n^-)$ induces a filtration on $V_q(\lambda)$ by letting
$$\mathcal{F}_k^\bold{d}V_q(\lambda):=\mathcal{F}_k^\bold{d}U_q(\n^-)\cdot\bold{v}_\lambda.$$
We let $V_q^\bold{d}(\lambda)$ denote the associated graded vector space: it is a cyclic $S_q(\lie n^-)$-module. Let $\bold{v}_\lambda^\bold{d}$ be the cyclic vector corresponding to $\bold{v}_\lambda$.

\section{Examples and properties of quantum degree cones}\label{SS:Rank2}\label{sec4}

\subsection{Examples of rank \texorpdfstring{$2$}{2}}
Before studying properties of these cones, we examine some small rank examples.

\begin{example}\label{Ss:QDC:A2} 
Let $\lie g = \mathfrak{sl}_3$ be the Lie algebra of type $\tt A_2$. For $\bold{d}\in\mathcal{D}$, let $d_{i,j}=\bold{d}(\alpha_{i,j})$.

We fix a sequence of positive roots $(\alpha_{1,1},\alpha_{1,2},\alpha_{2,2})$. The classical degree cone $\md$ is given by: 
$$\md=\{(d_{1,1},d_{1,2},d_{2,2})\in\mathbb{R}_{\geq 0}^{\Delta_+}\mid d_{1,1}+d_{2,2}>d_{1,2}\}.$$
We consider the quantum degree cones: $R(w_0)=\{s_1s_2s_1,\ s_2s_1s_2\}$. For the reduced decomposition $\underline{w}_0 = s_1s_2s_1$, let $F_{1,1},\ F_{1,2},\ F_{2,2}$ be the corresponding quantum PBW root vectors. The formula \eqref{Eq:LS} reads
$$F_{1,1} F_{2,2} = q^{-1}F_{2,2} F_{1,1} - q^{-1}F_{1,2},$$
implying $\md_{\underline{w}_0}^q=\md$. For $\underline{w}'_0 = s_2s_1s_2$, the same computation shows that $\mathcal{D}^q_{\!\!\underline{w}'_0}=\mathcal{D}$.
\end{example}

\begin{example}\label{Ss:QDC:C2}
Let $\g=\lie{sp}_4$ be the Lie algebra of type $\tt C_2$.
\par
For $\bold{d}\in\mathcal{D}$, let $d_{i,j}:=\bold{d}(\alpha_{i,j})$ and $d_{i,\overline{j}}:=\bold{d}(\alpha_{i,\overline{j}})$. The classical degree cone $\md$ is given by the following inequalities in $\mathbb{R}_{\geq 0}^{\Delta_+}$:
$$d_{1,1}+d_{2,2}>d_{1,2},\ \ d_{1,1}+d_{1,2}>d_{1,\overline{1}}.$$

\par
Fix a reduced decomposition $\underline{w}_0=s_1s_2s_1s_2$ of the longest element $w_0$ in the Weyl group of $\g$. Let 
$$F_{1,1},\ \ F_{1,\overline{1}},\ \ F_{1,2},\ \ F_{2,2}$$ 
be the corresponding quantum PBW root vectors, their commutation relations are:
$$F_{1,1}F_{1,\overline{1}}=q^2F_{1,\overline{1}}F_{1,1},\ \ F_{1,1}F_{1,2}=F_{1,2}F_{1,1}-(q+q^{-1})F_{1,\overline{1}},\ \ F_{1,1}F_{2,2}=q^{-2}F_{2,2}F_{1,1}-q^{-2}F_{1,2},$$
$$F_{1,\overline{1}}F_{1,2}=q^2F_{1,2}F_{1,\overline{1}},\ \ F_{1,\overline{1}}F_{2,2}=F_{2,2}F_{1,\overline{1}}+(1-q^{-2})F_{1,2}^{(2)},\ \ F_{1,2}F_{2,2}=q^2F_{2,2}F_{1,2}.$$
The quantum degree cone $\Dwq\subset \md$ is given by:
\begin{equation}\label{Eq:C2}
d_{1,1}+d_{2,2}>d_{1,2},\ \ d_{1,1}+d_{1,2}>d_{1,\overline{1}},\ \ d_{2,2}+d_{1,\overline{1}}>2d_{1,2}.
\end{equation}

The same construction with the reduced decomposition $\underline{w}_0' = s_2s_1s_2s_1$ shows that $\Dwq=\mathcal{D}^q_{\!\!\underline{w}_0'}$.
\end{example}

In the rank $2$ case, the quantum degree cone does not depend on the reduced decomposition.

\begin{proposition}\label{Prop:UniqueCone}
Let $\lie g$ be a simple Lie algebra of rank no more than $2$. For any $\underline{w}_0\in R(w_0)$, we have $\md^q=\mathcal{D}_{\underline{w}_0}^q$.
\end{proposition}

\begin{proof}
According to the above examples, it remains to consider the $\tt G_2$ case, which is given in Section \ref{Sec:G2}.
\end{proof}

\subsection{Properties of quantum degree cones}
The first property of the quantum degree cones we will prove  is the following:

\begin{theorem} \label{Thm:nullintersection}
Let $\lie g$ be a simple Lie algebra of rank $n \geq 3$, then
$$ \bigcap_{\underline{w}_0\in R(w_0)} \mathcal{D}_{\underline{w}_0}^q = \emptyset.$$
\end{theorem}

\begin{proof}
We show that there exist two reduced decompositions $\underline{w}^1_0,\underline{w}^2_0\in R(w_0)$ such that $\mathcal{D}_{\underline{w}_0^1}^q\cap \mathcal{D}_{\underline{w}_0^2}^q=\emptyset$. When $\g$ is a simple Lie algebra $\g$ of rank $3$, this is proved in Section \ref{Sec:A3}, \ref{Sec:B3} and \ref{Sec:C3} by explicit constructions.
\par
Let $\g$ be a simple Lie algebra of rank $>3$. There exists a Lie sub-algebra $\g'\subset\g$ of rank $3$ such that $\g'$ is a simple Lie algebra, we denote it by $\tt X_3$. The set of positive roots of $\tt X_3$ is denoted by $\Delta_+'$. We take $\underline{w}_L$ and $\underline{w}_L'$ as in the example of $\tt X_3$ in Section \ref{Sec:A3}, \ref{Sec:B3} or \ref{Sec:C3}, such that $\mathcal{D}_{\underline{w}_L}^q\cap \mathcal{D}_{\underline{w}_L'}^q=\emptyset$. Let $\underline{w}_0^1=\underline{w}_L\underline{w}_R$ and $\underline{w}_0^2=\underline{w}_L'\underline{w}_R'\in R(w_0)$. We claim that $\mathcal{D}_{\underline{w}_0^1}^q\cap \mathcal{D}_{\underline{w}_0^2}^q=\emptyset$. Let $p:\mathcal{D}_{\underline{w}_0^1}^q\ra \mathbb{R}_{\geq 0}^{\Delta_+'}$ be the restriction of functions. Then by definition,
$$p(\mathcal{D}_{\underline{w}_0^1}^q)=\mathcal{D}_{\underline{w}_L}^q,\ \ p(\mathcal{D}_{\underline{w}_0^2}^q)=\mathcal{D}_{\underline{w}_L'}^q.$$
This terminates the proof.
\end{proof}

This theorem implies that there is no degree function working for all reduced decompositions. In general, to study the relations between the cones associated to $\underline{w}_0$ and $\underline{w}_0'\in R(w_0)$ is a difficult task. But in some cases, the cone remains to be the same.
\par
Two reflections $s_p$ and $s_q$ in $W$ with $p\neq q$ are said to be \emph{orthogonal} if $s_ps_q=s_qs_p$. Two reduced decompositions $\underline{w}_0$, $\underline{w}_0'\in R(w_0)$ are said to be \emph{related by orthogonal reflections} if one can be obtained from the other by using only orthogonal reflections.

\begin{proposition}
Let $\underline{w}_0$, $\underline{w}_0'\in R(w_0)$ such that they are related by orthogonal reflections. Then $\mathcal{D}_{\underline{w}_0}^q=\mathcal{D}_{\underline{w}_0'}^q$.
\end{proposition}

\begin{proof}
By definition, it suffices to consider the case where
$$\underline{w}_0=\underline{w}_Ls_ps_q\underline{w}_R,\ \ \underline{w}_0'=\underline{w}_Ls_qs_p\underline{w}_R$$
with $1\leq p,q\leq n$ such that $s_ps_q=s_qs_p$. In this case, Lemma \ref{Lem:samecone} can be applied to finish the proof.
\end{proof}

\section{Local and global monomial sets}\label{sec5}

\subsection{Local monomial set for type \texorpdfstring{$\tt A_n$}{An}}
Let $\g=\lie {sl}_{n+1}$ be the Lie algebra of type $\tt A_n$.
\par
The following lemma gives an easy criterion to determine whether a degree is in the local monomial set.

\begin{proposition}\label{Prop:Anlm}
Let $\bold{d}\in S(\mathcal{D})$. The following statements are equivalent:
\begin{enumerate}
\item For any four different positive roots $\alpha$, $\beta$, $\gamma$, $\delta$ satisfying $\alpha+\beta=\gamma+\delta$, $d_\alpha+d_\beta\neq d_\gamma+d_\delta$; 
\item $\bold{d}\in \ms_{\lm}$.
\end{enumerate}
\end{proposition}

\begin{proof}
(1)$\Rightarrow$(2): Since in the $\tt A_n$ case, all fundamental representations are minuscule. The proof of \cite[Proposition 2]{FFR15} can be applied to show the validity of the hypothesis of Corollary \ref{Cor:useful}.\\
(2)$\Rightarrow$(1): Since $\g=\lie {sl}_{n+1}$, we can suppose that $\alpha=\alpha_{i,j}$, $\beta=\alpha_{k,l}$, $\gamma=\alpha_{i,l}$ and $\delta=\alpha_{j,k}$ with $i\leq k\leq j\leq l$. We consider the fundamental representation $V(\varpi_l)$: in $I(\varpi_l)$ there is a relation $f_{i,j}f_{k,l}v_{\varpi_l}\pm f_{i,l}f_{k,j}v_{\varpi_l}=0$. Since $I^\bold{d}(\varpi_l)$ is monomial, either $f_{i,j}f_{k,l}$ or $f_{i,l}f_{k,j}$ is in $I^\bold{d}(\varpi_l)$, which forbids the case $d_{\alpha_{i,j}}+d_{\alpha_{k,l}}=d_{\alpha_{i,l}}+d_{\alpha_{k,j}}$.
\end{proof}

\begin{example}
In general, the inclusions $\mathcal{D}_{\underline{w}_0}^q\subset\mathcal{S}_{\lm}$ and $\mathcal{S}_{\lm}\subset\mathcal{D}^q$ do not hold. Let $\mathfrak{g}$ be of type $\tt A_3$. The reduced decomposition $\underline{w}_0 = s_1s_2s_3s_2s_1s_2 \in R(w_0)$ induces the convex order on $\Delta_+$:
$$
\alpha_{1,1} < \alpha_{1,2} < \alpha_{1,3} < \alpha_{3,3} < \alpha_{2,3} < \alpha_{2,2}.
$$
We fix this sequence and identify $\mathbb{N}^{\Delta_+}$ with $\mathbb{N}^6$. Let $\bold{d} = (1,1,1,1,1,1)$, $\mathbf{d}'=(2,2,1,1,1,1)$ and $\bold{d}''=(1,1,1,1,1,2)$.  By Proposition \ref{Prop:Anlm}, $\bold{d}\notin\mathcal{S}_{\lm}$, $\bold{d}'$, $\bold{d}''\in\mathcal{S}_{\lm}$, but $\bold{d}$, $\bold{d}'\in\mathcal{D}_{\underline{w}_0}^q$, $\bold{d}''\notin\mathcal{D}^q$.
\end{example}

We show in the following example that $\mathcal{S}_{\gm}$ is in general a proper subset of $\mathcal{S}_{\lm}$.

\begin{example}
Let $\mathfrak{g}$ be of type $\tt A_n$ and let $\mathbf{d}$ be defined by $d_{\alpha_{i,j}} = 2^{(n-1) - (j - i)}$. It is clear that $\mathbf{d} \in \md$. If $\alpha_{i,j},\alpha_{k,l},\alpha_{i,l},\alpha_{k,j}$ are four different positive roots in $\Delta_+$ such that $\alpha_{i,j} + \alpha_{k,l} = \alpha_{i,l} + \alpha_{k,j}$, the indices must satisfy: $1\leq i<k\leq l<j\leq n$. In this case, we have $d_{\alpha_{i,j}} + d_{\alpha_{k,l}} > d_{\alpha_{i,l}} + d_{\alpha_{k,j}}$. Hence by Proposition \ref{Prop:Anlm} we have $\mathbf{d} \in \ms_{\lm}$.
\par
For arbitrary $1 \leq i \leq n$, let $P(\varpi_i)$ be the polytope obtained in \cite{BD}, such that its lattice points $S(\varpi_i):=P(\varpi_i)\cap\mathbb{N}^{\Delta_+}$ parametrizes a basis of $V(\varpi_i)$. Furthermore, by the choice of $\bd$, we have: 
$$S(\varpi_i)=\{ \bs \in \mathbb{N}^{\Delta_+} \mid f^{\bs}\cdot v_{\varpi_i}^{\bd} \neq 0 \mbox{ in } V^{\bd}(\varpi_i)\}.$$
But in general, for $\lambda=m_1\varpi_1+\ldots+m_n\varpi_n\in\mathcal{P}_+$, the Minkowski sum of lattice points $S(\varpi_1)^{+m_1}+\ldots+S(\varpi_n)^{+m_n}$ may not parametrize a basis of $V(\lambda)$. For instance, let $\g$ be of type $\tt A_4$, we have [\textit{loc.cit}]: $\#(S(\varpi_1)+S(\varpi_2)+S(\varpi_3)+S(\varpi_4))=1023$ but $\dim V(\varpi_1+\varpi_2+\varpi_3+\varpi_4)=1024$. Hence in general, $\bold{d}\notin\mathcal{S}_{\gm}$. 
\end{example}

\subsection{Global monomial sets: \texorpdfstring{$\tt A_n$}{An}}

We start with recalling the Dyck paths and FFLV polytopes \cite{FFL1}, \cite{FFL2}.
\par
A sequence ${\bf b}=(\delta_1,\dots,\delta_r)$ of 
positive roots is called a {\it Dyck path of type $\tt A_n$} if $\delta_1=\alpha_{i,i}$ and $\delta_r=\alpha_{j,j}$ for $i\leq j$ are
simple roots, and if $\delta_m=\alpha_{p,q}$, then $\delta_{m+1}=\alpha_{p+1,q}$ or $\delta_{m+1}=\alpha_{p,q+1}$.
\par
Let $A=\{1,2,\ldots,n,\overline{n-1},\ldots,\overline{1}\}$ be the totally ordered index set $1<2<\ldots<n<\overline{n-1}<\ldots<\overline{1}$. A {\it symplectic Dyck path} is a sequence ${\bf b}=(\delta_1,\dots,\delta_r)$ of positive roots (of $\lie {sp}_{2n}$) such that:
the first root is a simple root, $\beta_1=\alpha_{i,i}$; the last root is either a simple root $\beta_r= \alpha_{j,j}$ or
$\beta_r = \alpha_{j,\overline{j}}$ ($i \leq j \leq n$); if $\beta_m=\alpha_{r,q}$ with $r, q \in A$ then $\beta_{m+1}$ is
either $\alpha_{r,q+1}$ or $\alpha_{r+1,q}$, where $x+1$ denotes the smallest element in $A$ which is bigger than $x$.
\par
For a dominant integral weight $\lambda=\lambda_1\varpi_1+\lambda_2\varpi_2+\ldots+\lambda_n\varpi_n$ in the corresponding weight lattice, the FFLV polytopes $P_{\tt A_n}(\lambda)$ and $P_{\tt C_n}(\lambda)$ are defined by:

$$
\small
P_{\tt A_n}(\lambda)=\left\{ \mathbf m\in \mathbb R_{\ge 0}^{\Delta_+} \mid 
\begin{array}{c}
\hbox{for any $i=1,\ldots,n$ and any Dyck paths ${\bf b}=(\delta_1,\dots,\delta_r)$}\\ 
\hbox{starting in $\alpha_{i,i}$, ending in $\alpha_{j,j}: \sum_{\ell=1}^r m_{\delta_\ell}\le \lambda_i+\ldots+\lambda_j$}\\
\end{array}
\right\};
$$

$$
\small
P_{\tt C_n}(\lambda)=\left\{\mathbf m\in \mathbb R_{\ge 0}^{\Delta_+} \mid 
\begin{array}{c}
\hbox{for any $i=1,\ldots,n$ and any symplectic Dyck paths ${\bf b}=(\delta_1,\dots,\delta_r)$}\\ 
\hbox{starting in $\alpha_{i,i}$, ending in $\alpha_{j,j}: \sum_{\ell=1}^r m_{\delta_\ell}\le \lambda_i+\ldots+\lambda_j$};\\
\hbox{for any $i=1,\ldots,n$ and any symplectic Dyck paths ${\bf b}=(\delta_1,\dots,\delta_r)$ }\\ 
\hbox{starting in $\alpha_{i,i}$, ending in $\alpha_{j,\overline{j}}: \sum_{\ell=1}^r m_{\delta_\ell}\le \lambda_i+\ldots+\lambda_n$}
\end{array}
\right\}.
$$

Let $S_{\tt A_n}(\lambda)$ and $S_{\tt C_n}(\lambda)$ denote the set of lattice points in the corresponding polytopes. 
It has been shown in \cite{FFL1} and \cite{FFL2} that the polytopes satisfy for all $\lambda = \lambda_1 + \lambda_2$:
\begin{equation}\label{Eq:AC}
S_{\tt A_n}(\lambda) = S_{\tt A_n}(\lambda_1) + S_{\tt A_n}(\lambda_2) \text{ and } S_{\tt C_n}(\lambda) = S_{\tt C_n}(\lambda_1) + S_{\tt C_n}(\lambda_2) 
\end{equation}

Denote $d_{i,j}:=d_{\alpha_{i,j}}$ and $d_{i,\overline{j}}:=d_{\alpha_{i,\overline{j}}}$.
\par
For $\tt A_n$, consider $\bold{d}\in\md$ defined by: $d_{i,j}=(j-i+1)(n-j+1)$, then the following theorem has been proved in \cite{FFR15}:
\begin{theorem}
\begin{enumerate}
\item We have $\bold{d}\in\ms_{\gm}$. Moreover, let $\underline{w}_0=(s_n\ldots s_1)(s_n\ldots s_2)\ldots (s_ns_{n-1})s_n$, then $\bold{d}\in\Dwq$.
\item The set $\{f^\bold{a}\cdot v_\lambda^\bold{d}\mid \bold{a}\in S_{\tt A_n}(\lambda)\}$ forms a monomial basis of $V^\bold{d}(\lambda)$.
\end{enumerate}
\end{theorem}

\subsection{Global monomial sets: \texorpdfstring{$\tt C_n$}{Cn}}

Let us consider the $\tt C_n$ case and $\bold{d}\in\md$ defined by: 
$$d_{i,j} = (2n-j)(j-i+1),\ \ d_{i,\overline{j}} = j(2n - i - j +1).$$
This degree arises from an embedding of $\mathfrak{g}$ into a Lie algebra of type $\tt{A}_{2n-1}$.
We will show that $\mathbf{d} \in \ms_{\gm}$ and moreover, the monomial basis of $V^{\mathbf{d}}(\lambda)$ is parametrized by $S_{\tt C_n}(\lambda)$. For this we need an explicit description of the monomials associated to $S_{\tt C_n}(\varpi_k)$ from \cite{FFL2}:
\[
\{ f_{i_1, j_{\ell}-1} \cdots f_{i_{\ell}, j_{1}-1} \mid 1 < i_1 < \ldots < i_{\ell} \leq k  \leq j_{1} < \ldots< j_\ell   \}.
\]
\begin{lemma}
The degree function $\mathbf{d}\in\ms_{\lm}$ and for any fundamental weight $\varpi_k$, \[ \{f^\bold{a}\cdot v_{\varpi_k}^\bold{d}\mid \bold{a}\in S_{\tt C_n}(\varpi_k)\} \text{ is a basis of } V^\bold{d}(\varpi_k).\]
\end{lemma}
\begin{proof}
We need to show that the annihilating ideal of $ V^\bold{d}(\varpi_k)$ is monomial for all $\varpi_k$.  We start with the natural representation, namely the vector space $\mathbb{C}^{2n}$ with basis $\{e_1, \ldots, e_n, e_{\overline{n}}, \ldots, e_{\overline{1}}\}$ and operation $f_{i,j} e_{\ell} =   \delta_{i,\ell} c_{i,j}  e_{j+1}$ for some $c_{i,j} \in \mathbb{C}^*$ (when $j = \overline{p}$ we set $j+1 := \overline{p-1}$). We will further use that we can identify $V(\varpi_k)$ uniquely with a submodule in $\bigwedge^k \mathbb{C}^{2n}$. 
\par
First of all, since $\mathbf{d} \in \mathcal{D}$, we see that we can restrict ourselves to the nilpotent radical of the fundamental weight $\varpi_k$ (since all other root vectors are acting by $0$ on $v_{\varpi_k} \in V(\varpi_k)$ and hence on $V^{\mathbf{d}}(\varpi_k)$), \textit{e.g}. we have to consider monomials in $M_k := \{f_{i,j} \mid i \leq k \leq j \}$ only. We will prove the lemma in two steps: 
\begin{enumerate}
\item For any $i \leq k < j$, there exists a unique monomial $\mathbf{m}$  in the variables from $M_k$ with minimal degree such that $\mathbf{m}\cdot e_i = e_j$.
\item For any $j_1 < j_2 <  \ldots < j_k$ with $e_{j_1} \wedge \ldots \wedge e_{j_k} \in V(\varpi_k)$, there exists a  unique monomial $\mathbf{m}$ in the variables from $M_k$ with minimal degree such that $\mathbf{m}\cdot e_1 \wedge e_2 \wedge \ldots \wedge e_k  = e_{j_1} \wedge e_{j_2}  \wedge \ldots \wedge e_{j_k}$.
\end{enumerate}
Then the second step implies the Lemma.
\par
We start with proving $(1)$. Suppose $i \leq k \leq j \leq \overline{k}$, for weight reasons, there exists a unique monomial $\mathbf{m}$ in the variables from $M_k$ such that $\mathbf{m}\cdot e_i = e_j$, namely $\mathbf{m} = f_{\alpha_i,j-1}$.\\
Suppose $i \leq k  \leq n <  \overline{k} \leq \overline{p}$, and for simplicity we assume that $i \leq p$ (the $p \leq i$ case is similar). Let $\mathbf{m}$ be a monic monomial in the variables from $M_k$ such that $\mathbf{m}\cdot e_i = e_{\overline{p-1}}$, then for weight reasons $\mathbf{m}$ is in one of the following sets:
\begin{eqnarray}\label{eq:mon}
\{ f_{i, q-1}f_{p, \overline{q}} \, , \, f_{i,\overline{q}} f_{p,q-1} \mid q = k+1, \ldots, n\} \cup \{ f_{i, n} f_{p, n}\} \cup \{ f_{i, \overline{p}} \}.
\end{eqnarray}
We will see that  among these monomials, $f_{i, \overline{p}}$ is the unique monomial of minimal degree, namely of degree $p(2n - i - p +1)$.
\par
Let $i \leq q \leq n$ and denote:
$$
\begin{array}{cc}
\text{ for } p < q, & Y(q) :=  q(2n - i - q + 1)  + (2n - (q-1))(q-p); \\
\text{for } p \geq q, & Y(q) :=  q(2n - i - q + 1) + (2n-p)(p-q+2); \\
\text{for any } p, & X(q): = (2n-(q-1))((q-1) - i +1) + q(2n-p-q+1).
\end{array}
$$
We have:
$$
\begin{array}{ccc}
\text{if } p \leq q, & \text{then} & \operatorname{deg }_{\bd}(f_{i,q-1}f_{p, \overline{q}}) = X(q);\\
\text{if } p > q,& \text{then} & \operatorname{deg }_{\bd}  (f_{i,q-1}f_{q, \overline{p}}) = X(q);\\
\text{ if } p < q, & \text{then}  & \operatorname{deg}_{\bd} (f_{i,\overline{q}}f_{p, q-1} )= Y(q);\\
\text{ if } p \geq q, & \text{then} & \operatorname{deg}_{\bd}( f_{i,\overline{q}}f_{q-1, p}) = Y(q).\\
\end{array}
$$
Now it is straightforward to see that for $q > i$:  
\[
X(q) > X(q-1) \;\;  \text{ and moreover } \; \;  X(i) > p(2n - i - p +1),
\]
as well as
\[
Y(q) > Y(q+1) \text{  and  } Y(n) = X(n).
\]
Combining both gives
\[
Y(i) >  \ldots > Y(n) = X(n) >  \ldots > X(i) > p(2n - i - p +1) = \operatorname{deg}_{\bd} f_{i, \overline{p}}.
\]
Moreover 
\[
\operatorname{deg }_{\bd} (f_{i,n}f_{p,n}) = n(2n-i-p+2) > p(2n-i-p+1) = p(2n - i - p +1) = \operatorname{deg}_{\bd} f_{i, \overline{p}}.
\]
This implies, that  $f_{i, \overline{p}}$ is the unique monomial of minimal degree among all monomials in \eqref{eq:mon} and the first step is done.\\
\par

We are left with step $(2)$. Let $e_{i_1} \wedge \ldots \wedge e_{i_k} \in V(\varpi_k) \subset \bigwedge^k \mathbb{C}^{2n}$, with $i_1 < \ldots < i_k$ and $i_j \in \{ 1, \ldots, \overline{1}\}$. Let $\mathbf{m}$ be a monomial of minimal degree in the variables from $M_k$ such that 
\[
\mathbf{m}\cdot e_1 \wedge \ldots \wedge e_k = e_{i_1} \wedge \ldots \wedge e_{i_k} + \text{rest}.
\]
Due to the operation on the tensor product, there exists a factorization $\mathbf{m} = \prod_{\ell=1}^k \mathbf{m}_\ell$ and a permutation $\sigma \in \mathfrak{S}_k$,  such that $\mathbf{m}_\ell\cdot e_\ell = e_{i_{\sigma(\ell)}}$. 
Since $\mathbf{m}$ is in the variables from $M_k$ only, we see that if $ \ell \in \{i_1, \ldots, i_k\} \cap \{1, \ldots, k\}$, then $\mathbf{m}_\ell = 1$ and hence $i_\sigma(\ell) = \ell$. So without loss of generality, we may assume that $k < i_1 < i_2 < \ldots < i_k$.
\par
Suppose now there exist $\ell < j $ with $\sigma(\ell) < \sigma(j)$. We have
\[
\mathbf{m}_\ell \mathbf{m}_j\cdot e_\ell \wedge e_j = e_{i_{\sigma(\ell)}} \wedge e_{i_{\sigma(j)}} + \text{ rest }.
\]
From step $(1)$ we deduce that if $\mathbf{m}$ is of minimal degree, then
\[
\mathbf{m}_\ell = f_{\ell, i_{\sigma(\ell)}-1} \; , \; \mathbf{m}_j = f_{j, i_{\sigma(j)}-1}.
\]
Similarly to the $\tt{A}_n$ considerations (recall that the $\tt{C}_n$-degree is a obtain from an $\tt{A}_{2n-1}$-degree), we see that 
\[
\operatorname{deg}_{\bd}  (f_{\ell, i_{\sigma(j)}-1} f_{j, i_{\sigma(\ell) }-1}) < \operatorname{ deg }_{\bd}  (f_{\ell, i_{\sigma(\ell) }-1} \;  f_{j, i_{\sigma(j) }-1}).
\]
We denote  $\mathbf{m}' :=   f_{\ell, i_{\sigma(j)}-1} f_{j, i_{\sigma(\ell) }-1} \left( \prod_{i \neq j ,\ell} \mathbf{m}_i  \right) $, then 
\[
\operatorname{ deg }_{\bd} \mathbf{m} > \operatorname{ deg }_{\bd}  \mathbf{m}'.
\]
But by construction:
\[
\mathbf{m}'\cdot e_1 \wedge \ldots \wedge e_k =  e_{i_1} \wedge \ldots \wedge e_{i_k} + \text{rest},
\]
we have a contradiction to the minimality of the degree of $\mathbf{m}$ and hence $\sigma(\ell) > \sigma(j)$ for all $\ell < j $. 
\par
Let $\{i_1 \ldots, i_k\} = \{p_1 < \ldots < p_s\} \cup \{ \ell_1 < \ldots < \ell_{k-s}\}$, where $\ell_{k-s} \leq k < p_1$ and $\{q_1 < \ldots < q_s\}$ be the complement of $ \{ \ell_1 < \ldots < \ell_{k-s}\}$ in $\{1, \ldots, k\}$, then the monomial of minimal degree to obtain $e_{i_1} \wedge \ldots \wedge e_{i_k}$ is
\[
f_{q_1, p_s-1} \cdots f_{q_s, p_1 - 1}.
\]
This proves that $\mathbf{d} \in  \mathcal{S}_{\lm}$ and moreover these are precisely the monomials associated to $S_{\tt{C}_n}(\varpi_k)$.
\end{proof}

From here we can deduce by using \eqref{Eq:AC} and Theorem~\ref{Thm:Minkow}:
\begin{theorem}
\begin{enumerate}
\item For the degree function: $\bold{d}\in\ms_{\gm}$.
\item The set $\{f^\bold{a}\cdot v_\lambda^\bold{d}\mid \bold{a}\in S_{\tt C_n}(\lambda)\}$ forms a monomial basis of $V^\bold{d}(\lambda)$.
\item If  $\bd \in \ms_{\gm}$ and the corresponding monomial basis is associated to $S_{\tt C_n}(\lambda)$, then $\bold{d}\notin\md^q$.
\end{enumerate}
\end{theorem}
\begin{proof}
Part (1) and (2) are deduced from the lemma, by using \eqref{Eq:AC} and Theorem~\ref{Thm:Minkow}. It is left to prove the part $(3)$, i.e. we assume that $\bd \in \md$ satisfies (1) and (2) and we want to show $\bd \notin \md^q$. We consider the simple Lie subalgebra $\mathfrak{g}_2$ of type $\tt{C}_2$ in $\mathfrak{g}$ with positive roots $\alpha_{n-1,n-1}, \alpha_{n-1, \overline{n-1}}, \alpha_{n-1,n}$ and $\alpha_{n,n}$. In the subalgebra $U_q(\mathfrak{g}_2) \subset U_q(\mathfrak{g})$ we have the following relation, independent of the chosen reduced expression (see Example~\ref{Ss:QDC:C2}),
$$ F_{n-1, \overline{n-1}} F_{n,n} = F_{n,n}F_{n-1, \overline{n-1}} + (1 - q^{-2}) F_{n-1,n}^{(2)},$$
implying that every $\mathbf{d}' \in \mathcal{D}^q$ satisfies $\bd'_{n-1, \overline{n-1}} + \bd'_{n,n} > 2 \bd'_{n-1,n}$.
\par
Since $\mathbf{d}$ satisfies (1) and (2), in $V^{\mathbf{d}}(\varpi_{n})$ we have $f_{n-1, \overline{n}} f_{n,n}\cdot v^{\mathbf{d}}_{\varpi_{n}} \neq 0$ and $f^2_{n-1,n}\cdot v^{\mathbf{d}}_{\varpi_{n}} = 0$ which implies $\bd_{n-1, \overline{n-1}} + \bd_{n,n} < 2 \bd_{n-1,n}$. Hence $\mathbf{d} \notin \mathcal{D}^q$.
\end{proof}

\begin{remark}
If $\ms_{\gm}\neq\emptyset$, then there exists an $\mathbb{N}$-filtration arising from $\bold{d}\in\ms_{\gm}$ such that for any $\lambda\in\mathcal{P}_+$, $V^\bold{d}(\lambda)$ has a unique monomial basis. 
\par
If $\bd \in \Dwq \cap \ms_{\gm}$, then, by the argument in \cite[Theorem 5]{FFR15}, there exists an $\mathbb{N}$-filtration on $U_q(\lie n^-)$ arising from $\bd$ such that for any $\lambda\in\mathcal{P}_+$, $V^\bold{d}_q(\lambda)$ has a unique monomial basis in $S_q(\lie n^-)$.
\end{remark}

\subsection{Global monomial set: \texorpdfstring{$\tt C_2$}{C2}}
Consider the quantum degree cone $\Dwq$ defined in (\ref{Eq:C2}). We pick a solution such that the sum $a_1+a_2+a_3+a_4$ takes its minimal value: 
$$\bold{d}=(d_{1,1},d_{1,\overline{1}},d_{1,2},d_{2,2})=(1,1,1,2).$$
Since $\bold{d}\in\md$, we consider the induced degree on the enveloping algebra with PBW root vectors $f_{1,1}$, $f_{1,\overline{1}}$, $f_{1,2}$ and $f_{2,2}$.
 
\begin{proposition}\label{lem:monom}
We have $\bold{d}\in\ms_{\lm}$, \emph{i.e.}, the defining ideals $I^\bold{d}(\varpi_1)$ and $I^\bold{d}(\varpi_2)$ are monomial.
\end{proposition}
The proof of the proposition is omitted, as it is a straightforward computation with the help of Corollary~\ref{Cor:useful} (1). 
\par
We turn to study whether $\bold{d}$ is in the global monomial set $\ms_{\gm}$.
\par
Let $\bold{SP}_{4}(m_1,m_2)\subset\mathbb{R}^4$ be the polytope defined by the following inequalities:
$$x_1,\ x_2,\ x_3,\ x_4\geq 0,\ \ x_1\leq m_1,\ \ x_4\leq m_2,$$
$$2x_1+x_2+2x_3+2x_4\leq 2(m_1+m_2),\ \ x_1+x_2+x_3+2x_4\leq m_1+2m_2.$$
Let $S(m_1,m_2)$ denote the lattice points in $\bold{SP}_4(m_1,m_2)$.

\begin{theorem}\label{Thm:C2}
For any $\lambda=m_1\varpi_1+m_2\varpi_2\in\mathcal{P}_+$, the following statements hold:
\begin{enumerate}
\item The set $\{f^{\bold{p}} v_\lambda^\bold{d}\mid\bold{p}\in S(m_1,m_2)\}$ forms a basis of $V^\bold{d}(\lambda)$, hence a basis of $V(\lambda)$.
\item We have $\bold{d}\in\ms_{\gm}$, \emph{i.e.}, the defining ideal $I^\bold{d}(\lambda)$ is monomial.
\end{enumerate}
\end{theorem}

The rest of this paragraph will be devoted to prove this theorem.

\begin{proposition}\label{Prop:MinkC}
For any $m_1,m_2,m_1',m_2'\in\mathbb{N}$, 
$$S(m_1,m_2)+S(m_1',m_2')=S(m_1+m_1',m_2+m_2').$$
\end{proposition}
\begin{proof}
It suffices to prove that for $m_1>0$ and $m_2\geq 0$,
$$S(m_1-1,m_2)+S(1,0)=S(m_1,m_2) \text{ and } S(0,m_2-1) + S(0,1) = S(0,m_2).$$
First suppose that $m_1 \neq 0$ and pick $\bold{s} = (a_1, a_2, a_3, a_4) \in S(m_1, m_2)$. 
\begin{enumerate}
\item If $a_1 \neq 0$, we set $\bold{t}_1 = (a_1 - 1, a_2, a_3, a_4)$ and $\bold{t}_2 = (1,0,0,0)$; then $\bold{t}_2 \in S(1,0)$. Since $\bold{s} \in S(m_1, m_2)$, $2a_1 + a_2 + 2a_3 + 2a_4 \leq 2(m_1+m_2)$ implies that $2(a_1-1) + a_2 + 2a_3 + 2a_4 \leq 2(m_1 -1 + m_2)$;
$a_1 + a_2 + a_3 + 2a_4 \leq m_1  +2m_2$ implies that $a_1 - 1 + a_2 + a_3 + 2a_4 \leq (m_1 - 1) + 2m_2$. Combining them together we get $\bold{t}_1 \in S(m_1-1, m_2)$.
\item If $a_1 = 0$ and $a_3 \neq 0$, the very similar argument with $\bold{t}_2 = (0,0,1,0)$ implies again $\bold{t}_1=\bold{s}-\bold{t}_2 \in S(m_1-1, m_2)$.
\item Suppose that $a_1 = 0$, $a_3 = 0$ but $a_2 \neq 0$. The inequalities for $\bold{s} \in S(m_1, m_2)$ are reduced to $a_2 + 2a_4 \leq m_1 + 2m_2$. We see that $\bold{s} = (0,a_2-1, 0, a_4) + (0,1,0,0)$ gives a decomposition in $S(m_1-1, m_2) + S(1,0)$. 
\item When $a_1 =a_2=a_3= 0$ but $a_4 \neq 0$, the decomposition is obvious.
\end{enumerate}
Suppose now $m_1 = 0$ and pick $\bold{s} = (a_1, a_2, a_3, a_4)  \in S(0,m_2)$. Then $a_1 = 0$ and the inequality $a_2 + a_3 + 2a_4 \leq 2m_2$ is redundant. 
\begin{enumerate}
\item Suppose $a_3 \neq 0$, then we decompose $\bold{s} = (0,a_2 -1, a_3, a_4) + (0,1,0,0)$: it is clear $(0,1,0,0)\in S(0,1)$; since $a_2 + 2a_3 + 2a_4 \leq 2m_2$, we get $a_2 + 2(a_3 - 1) + 2a_4 \leq 2(m_2-1)$, it implies $(0,a_2 -1, a_3, a_4)\in S(0,m_2-1)$.
\item The case $a_1 =a_3= 0$ but $a_4 \neq 0$ can be dealt in a similar way.
\item We are left with the case where $0 \neq a_2 \leq 2m_2$. If $a_2\leq 2$, there is nothing to be shown; if $a_2>2$, we decompose it as $(0,a_2-2, 0,0) + (0,2,0,0)$.
\end{enumerate}
Repeating this procedure shows that any element in $S(m_1, m_2)$ can be decomposed as the sum of elements in $S(m_1 - k, m_2 - \ell)$ and in $S(k,\ell)$.
\end{proof}

To apply Theorem \ref{Thm:Minkow} to terminate the proof of Theorem \ref{Thm:C2}, it suffices to count the number of lattice points in $\bold{SP}_4(m_1,m_2)$.
\par
For any integers $a,b\in\mathbb{N}$, we define a polytope $\bold{P}(a,b)\subset\mathbb{R}^2$ by the following inequalities:
$$x\geq 0, \ \ y\geq 0,\ \ x+2y\leq a,\ \ x+y\leq b.$$
Let $N(a,b)$ denote the number of lattice points in $\bold{P}(a,b)$.

\begin{lemma}\label{Lem:Pcount}
The number of lattice points $N(a,b)$ has the following expression:
\begin{enumerate}
\item $N(a,a)=\left\{\begin{matrix} l(l+1) & \text{if } a=2l-1;\\ (l+1)^2 & \text{if }a=2l.\end{matrix}\right.$
\item $N(a,b)=\left\{\begin{matrix}
N(a,a), & \text{if }b\geq a; \\
\frac{1}{2}(b+1)(b+2), & \text{if } a\geq 2b;\\
-l^2+2lb-\frac{1}{2}b^2+\frac{1}{2}b+l+1, & \text{if } 2b>a>b \text{ and }a=2l;\\
-l^2+2lb-\frac{1}{2}b^2+\frac{3}{2}b+1, & \text{if } 2b>a>b \text{ and }a=2l+1.
\end{matrix}\right.$
\end{enumerate}
\end{lemma}

\begin{proof}
It amounts to count the integral points in the closed region cutting by the lines $x+2y=a$, $x+y=b$ and the two axes in $\mathbb{R}^2$ which depends on the position of the intersection of these two lines.
\end{proof}

\begin{proposition}
The number of lattice points in $\bold{SP}_4(m_1,m_2)$ is
$$\frac{1}{6}(m_1+1)(m_2+1)(m_1+m_2+2)(m_1+2m_2+3).$$
\end{proposition}

\begin{proof}
Let $H$ be the intersection of hyperplanes $x_1=\alpha$ and $x_4=\beta$ in $\mathbb{R}^4$ with coordinates $(x_1,x_2,x_3,x_4)$ where $\alpha,\beta\geq 0$. By definition, 
$$H\cap \bold{SP}_4(m_1,m_2)=\bold{P}(2m_1+2m_2-2\alpha-2\beta,m_1+2m_2-\alpha-2\beta).$$
Therefore by Lemma \ref{Lem:Pcount}, the number of integral points in $\bold{SP}_4(m_1,m_2)$ equals
\begin{equation}\label{Eq:N(a,b)}
\sum_{\alpha=0}^{m_1}\sum_{\beta=0}^{m_2}N(2m_1+2m_2-2\alpha-2\beta,m_1+2m_2-\alpha-2\beta).
\end{equation}
Since $\alpha\leq m_1$ and $\beta\leq m_2$, it falls into the third case in Lemma \ref{Lem:Pcount} (2) and (\ref{Eq:N(a,b)}) reads (where $l=m_1+m_2-\alpha-\beta$ and $b=m_1+2m_2-\alpha-2\beta$):
\begin{small}
$$
\sum_{\alpha=0}^{m_1}\sum_{\beta=0}^{m_2}\frac{1}{2}\alpha^2+2\alpha\beta+\beta^2-(m_1+2m_2+\frac{3}{2})\alpha-2(m_1+m_2+1)\beta+(\frac{1}{2}m_1^2+2m_1m_2+m_2^2+\frac{3}{2}m_1+2m_2+1).
$$
\end{small}
An easy summation provides the number in the statement.
\end{proof}

By Weyl character formula, for $\lambda=m_1\varpi_1+m_2\varpi_2\in\mathcal{P}_+$, $\dim V(\lambda)$ coincides with the number of lattice points in $\bold{SP}_4(m_1,m_2)$. This terminates the proof of Theorem \ref{Thm:C2}.

\begin{remark}
Up to permuting the second and the third coordinates, the polytope $\bold{SP}_4(m_1,m_2)$ coincides with the one in Proposition 4.1 of \cite{Kir1} (see also \cite{Kir2}), which is unimodularly equivalent to the Newton-Okounkov body of some valuation arising from inclusions of (translated) Schubert varieties.
\end{remark}

There are several other known polytopes parameterizing bases of a finite dimensional irreducible representation $V(\lambda)$ of $\lie{sp}_4$. For example, the Gelfand-Tstelin polytope $P_1(\lambda)$ \cite{BZ01}; the FFLV polytope $P_2(\lambda)$ \cite{FFL2}; the string polytope $P_3(\lambda)$ associated to the reduced decomposition $\underline{w}_0=s_1s_2s_1s_2$ \cite{Lit98}; the string polytope $P_4(\lambda)$ associated to the reduced decomposition $\underline{w}_0=s_2s_1s_2s_1$ [\emph{loc.cit.}]; when $\lambda=m_1\varpi_1+m_2\varpi_2$, the polytope $\bold{SP}_4(m_1,m_2)$.
\par
With the help of Polymake \cite{GJ97}, one can verify that: the polytopes $P_1(\lambda)$, $P_2(\lambda)$ and $P_4(\lambda)$ are unimodular equivalent; but the polytope $P_3(\lambda)$ and $\bold{SP}_4(m_1,m_2)$ are not unimodular equivalent to any other polytopes.

\begin{remark}
Using the polyhedral cones associated to these polytopes, the construction in \cite{FaFoL15} can be applied to produce three non-isomorphic toric degenerations of the spherical varieties associated to the symplectic group $\text{Sp}_4$, see for example \cite[Section 10, Section 12, Section 13]{FaFoL15}.
\end{remark}

\subsection{Global monomial set: $\tt D_4$} 
We prove that the global monomial set for $\tt D_4$ is non-empty. We refer to Section \ref{Sec:D4} for details on the cones and the enumeration of positive roots. Let $\bold{d}=(5,5,1,2,4,1,1,2,6,10,12,20)$. It is shown in Section \ref{Sec:D4} that there exists a $\underline{w}_0\in R(w_0)$ such that $\bold{d}\in\mathcal{D}_{\underline{w}_0}^q$. We will freely use the notations in Section \ref{Sec:D4}.

\begin{proposition} 
We have $\bd \in \mathcal{S}_{\lm}$. 
\end{proposition}
Again, we omit the proof as it is straightforward with  Corollary~\ref{Cor:useful}.

\begin{theorem}
\begin{enumerate}
\item We have $\bold{d}\in\ms_{\gm}$. 
\item The set $\{f^\bold{s}\cdot v_\lambda^\bold{d}\mid \bold{s}\in S_{\tt D_4}(\lambda)\}$ forms a monomial basis of $V^\bold{d}(\lambda)$.
\item Let $\underline{w}_0 = s_2s_1s_2s_3s_2s_4s_2s_1s_2s_3s_2s_4 \in R(w_0)$, then $\bd \in \md_{\underline{w}_0}^q$.
\end{enumerate}
\end{theorem}

\begin{proof}
Let $P_{\tt D_4}(\lambda)$ be the polytope defined in \cite[Section 3]{Gor2} and $S_{\tt D_4}(\lambda)$ be the set of lattice points in $P_{\tt D_4}(\lambda)$. By a straightforward comparation, we obtain $$S_{\tt D_4}({\varpi_i}) = \{ \bs \in \N^{\Delta_+} \mid f^{\bs}\cdot v_{\varpi_i}^{\bd} \neq 0 \mbox{ in } V^\bd(\varpi_i)\},\ \ i = 1,2,3,4.$$
It is shown in [\emph{loc.cit.}] that for any $\lambda,\mu \in \mathcal{P}_+$, we have:
$$P_{\tt D_4}(\lambda) + P_{\tt D_4}(\mu) = P_{\tt D_4}(\lambda+\mu)\ \mbox{ and }\ S_{\tt D_4}(\lambda) + S_{\tt D_4}(\mu) = S_{\tt D_4}(\lambda+\mu)$$ and $\dim V(\lambda) = \#S_{\tt D_4}(\lambda)$.
\par
The statements (1) and (2) follow from Theorem \ref{Thm:Minkow}. The part (3) is shown in Section \ref{Sec:D4}.
\end{proof}

\subsection{Global monomial set: $\tt B_3$} 

Let $\lie g$ be of type $\tt B_3$. For $\lambda\in\mathcal{P}_+$, let $P_{\tt B_3}(\lambda)$ denote the polytope defined in \cite[Section 5]{BK} and $S_{\tt B_3}(\lambda)$ be the set of lattice points in $P_{\tt B_3}(\lambda)$.
\par
Let $f_{i,j}$ and $f_{i,\overline{j}}$ be the PBW root vectors associated to the positive roots $\alpha_{i,j}$ and $\alpha_{i,\overline{j}}$ respectively. For $\bold{d}\in\mathbb{R}_{\geq 0}^{\Delta_+}$, we write $d_{i,j}=\bold{d}(\alpha_{i,j})$ and $d_{i,\overline{j}}=\bold{d}(\alpha_{i,\overline{j}})$. We consider the element $\bold{d}\in\mathbb{R}_{\geq 0}^{\Delta_+}$ defined by:
$$ d_{1,1} = 4, \ d_{1,2} = 3,\ d_{2,2} = 3,\ d_{1,3} = 3,\ d_{1,\overline{2}} = 1,\ d_{1,\overline{3}} = 1,\ d_{2,3} = 4,\ d_{2,\overline{3}} = 3,\ d_{3,3} = 2.$$ 
We will show in Section \ref{Sec:B3} that $\bold{d}\in\mathcal{D}$.

\begin{theorem}
\begin{enumerate}
\item We have $\bold{d}\in\ms_{\gm}$, and the set $\{f^\bold{s}\cdot v_\lambda^\bold{d}\mid \bold{s}\in S_{\tt B_3}(\lambda)\}$ forms a monomial basis of $V^\bold{d}(\lambda)$.
\item For any $\bold{e}\in \md$ satisfying (1), we have $\bold{e} \notin \mathcal{D}^q$.
\end{enumerate}
\end{theorem}

\begin{proof}
\begin{enumerate}
\item As before by computing each weight space in $V^{\bd}(\varpi_i), i= 1,2,3$, we obtain $\bd \in \mathcal{S}_{\lm}$. By comparing the basis arising from the monomiality of the defining ideals $I^{\bd}(\varpi_i)$ with the basis obtain in [\emph{loc.cit.}], we get for $i=1,2,3$: 
$$S_{\tt B_3}({\varpi_i}) = \{ \bs \in \N^{\Delta_+} \mid f^{\bs}\cdot v_{\varpi_i}^{\bd} \neq 0 \mbox{ in } V^\bd(\varpi_i)\}.
$$
For any $\lambda,\mu \in \mathcal{P}_+$, we have: 
$$P_{\tt B_3}(\lambda) + P_{\tt B_3}(\mu) = P_{\tt B_3}(\lambda+\mu),\ \ S_{\tt B_3}(\lambda) + S_{\tt B_3}(\mu) = S_{\tt B_3}(\lambda+\mu)$$ and $\dim V(\lambda) = \# S_{\tt B_3}(\lambda)$. By Theorem \ref{Thm:Minkow}, $\bold{d}\in\mathcal{S}_{\gm}$. 

\item Let $\bold{e} \in \mathcal{S}_{\gm}$. From reading the lattice points in $P_{\tt B_3}(\varpi_2)$, we get: $f_{1,2}f_{1,\overline{3}}\cdot v^{\bold{e}}_{\varpi_2} \neq 0$ in $V^{\bold{e}}(\varpi_2)$. Since the corresponding weight space is one-dimensional and $f_{1,3}^2$ has the same weight, $f_{1,3}^2\cdot v^{\bold{e}}_{\varpi_2} = 0$.
\par
Assume $\underline{w}_0\in R(w_0)$ such that $\bold{e} \in \Dwq$. Let $<$ be the induced convex order on $\Delta_+$. Without loss of generality we suppose that $\alpha_{1,2}< \alpha_{1,\overline{3}}$. 
\begin{enumerate}[\textbf{Case} 1]
\item Assume $\alpha_{1,2}<\alpha_{1,3}< \alpha_{1,\overline{3}}$, for the quantum degree cone $\Dwq$, by computing the L-S formula explicitly, this would imply the following inequality:  $d_{1,2} + d_{1,\overline{3}} > 2 d_{1,3}$.
This implies, turning to the classical case, $f_{1,3}^2\cdot v_{\varpi_2}^{\bold{e}} \neq 0$, which is a contradiction.
\item Assume $\alpha_{1,3}<\alpha_{1,2}<\alpha_{1,\overline{3}}$. Consider the root $\alpha_{3,3}$: by the convexity, it must be simultaneously larger than $\alpha_{1,\overline{3}}$ and smaller than $\alpha_{1,2}$. This is a contradiction.
\item Assume $\alpha_{1,2}<\alpha_{1,\overline{3}}<\alpha_{1,3}$, with similar arguments as in Case 2 we get a contradiction.
\end{enumerate}
As a conclusion, for any $\underline{w}_0\in R(w_0)$, $\bold{e} \notin \Dwq$.
\end{enumerate}
\end{proof}

\subsection{Global monomial set: $\tt G_2$}
Let $\lie g$ be of type $\tt G_2$. We use the notations in Section \ref{Sec:G2}. Consider the following $\bold{d}\in\mathcal{D}$:
$$d_{1} = 2,\ \ d_{1112} = 1,\ \ d_{112} = 3,\ \ d_{11122} = 1,\ \ d_{12} = 3,\ \ d_2 = 2.$$ 
It is clear that $\bd \in \mathcal{D}$.
\par
Let $P_{\tt G_2}(\lambda)$ be the polytope defined in \cite[Section 1]{Gor1}, the set of its lattice points will be denoted by $S_{\tt G_2}(\lambda)$. With similar arguments and calculations as before we obtain the first statement of the following theorem. The second statement follows from Section \ref{Sec:G2}, there we show, for each $\mathbf{e} \in \mathcal{D}^q$ there exist a unique monomial basis of $V^{\mathbf{e}}(\varpi_i), i=1,2$, i.e. $\mathbf{e} \in \ms_{\lm}$, which does not coincide with the basis in (1) of the following theorem.

\begin{theorem}\label{Th:gmG2} 
\begin{enumerate}
\item We have $\bold{d}\in\ms_{\gm}$, and the set $\{f^\bold{s}\cdot v_\lambda^\bold{d}\mid \bold{s}\in S_{\tt G_2}(\lambda)\}$ forms a monomial basis of $V^\bold{d}(\lambda)$.\par
\item For all $\bd \in \md$ satisfying (1), we have $\bd \notin \mathcal{D}^q$.
\end{enumerate}
\end{theorem}

\subsection{Local monomial sets: \texorpdfstring{$\tt G_2$}{G2}}\label{LMS:G2}
Let $\lie g$ be of type $\tt G_2$. By Proposition \ref{Prop:UniqueCone}, the quantum degree cone $\Dwq$ does not depend on the choice of $\underline{w}_0\in R(w_0)$. Let 
$$\bold{d}=(d_1,d_{1112},d_{112},d_{11122},d_{12},d_{2})=(2,2,1,2,2,5).$$ 
We will show in Section \ref{Sec:G2} that $\bold{d}\in \Dwq$. Let $f_1, f_{1112}, f_{112}, f_{11122}, f_{12}, f_{2}$ be the PBW root vectors in $\lie n^-$.

\begin{proposition}\label{Lem:MIG2}
We have $\bold{d}\in \mathcal{S}_{\lm}$, i.e., the defining ideals $I^{\bf d}(\varpi_1)$ and $I^{\bf d}(\varpi_2)$ are monomial.
\end{proposition}
We omit the proof as before.
\par
We want to examine that this degree function is quite interesting, due to the fact that the induced semigroup is not saturated as we will explain:
\par
Let $S(\varpi_2) = \{\bs \in \N^{\Delta_+} \mid f^{\bs}\cdot v^{\bd}_{\varpi_2} \neq 0\}$. We have by construction $\#S(\varpi_2) = \dim V(\varpi_2) = 14$, but there are $16$ lattice points in the convex hull $P = \mbox{conv}(S(\varpi_2))$.
\par
We fix the sequence of positive roots $(\alpha_1,\alpha_{1112},\alpha_{112},\alpha_{11122},\alpha_{12},\alpha_2)$ to identify $\mathbb{R}^{\Delta_+}$ and $\mathbb{R}^6$. Let $\bold{G_2^{\varpi_1}}(m_1) \subset \mathbb{R}^6$ be the polytope defined by the inequalities:
$$x_1,\ x_2,\ x_3,\ x_4,\ x_5,\ x_6\geq 0,\ \ x_1\leq m_1,\ \ x_6\leq 0,\ \ 2x_1+2x_2+x_3+2x_4+2x_5\leq 2m_1.$$

Let $\bold{G_2^{\varpi_2}}(m_2) \subset \mathbb{R}^6$ be the polytope defined by the inequalities:
$$x_1,\ x_2,\ x_3,\ x_4,\ x_5,\ x_6\geq 0,\ \ x_1\leq 0,\ \ 2x_2+x_3+x_4+x_5+2x_6\leq 2m_2.$$
Let $\{e_1,e_2,\ldots,e_6\}$ be the standard basis of $\mathbb{R}^6$.
\begin{conjecture}
Let $\lambda = m_1 \varpi_1 + m_2 \varpi_2 \in \mathcal{P}_+$. The number of lattice points in the Minkowski sum $$m_1 \bold{G_2^{\varpi_1}}(1) + m_2 \left(\bold{G_2^{\varpi_2}}(1) \cup \{3e_3,3e_5\}\right)$$ coincides with $\dim V(m_1 \varpi_1 + m_2\varpi_2)$.
\end{conjecture}

\begin{remark} Note that the proof of Lemma \ref{Lem:MIG2} does not depend on the choice of $\mathbf{d} \in \Dwq$. Further we have $\mathcal{D}^q = \Dwq$ (see Proposition \ref{Prop:UniqueCone}). This implies the inclusion $\mathcal{D}^q \subset \mathcal{S}_{\lm}$.
\end{remark}

\begin{remark}
Let $G$ be the complex algebraic group of type $\tt G_2$ and $U$ (resp. $U^-$) be the maximal unipotent subgroup of $G$ having $\lie n^+$ (resp. $\lie n^-$) as Lie algebra.
\par
Let $S=(\alpha_1,\alpha_{1112},\alpha_{112},\alpha_{11122},\alpha_{12}, \alpha_2)$ be a birational sequence for $U^-$ (see \cite[Section 3]{FaFoL15} for definition). Using $S$ we identify $\mathbb{N}^{\Delta_+}$ and $\mathbb{N}^6$. We fix the integral weight function $\Psi:\Delta_+\ra\mathbb{N}$ by $\Psi=\bold{d}$ and the lexicographic order on $\mathbb{N}^6$. Let $>$ be the total order on $\mathbb{N}^6$ defined in \cite[Section 5]{FaFoL15} by combining $\Psi$ and the lexicographic order. In [\textit{loc.cit}], a monoid $\Gamma=\Gamma(S,>)\subset \mathcal{P}_\mathbb{R}\times\mathbb{N}^6$ is attached to $G/\hskip -3.3pt/U$ to study its toric degenerations. Let $\pi_1:\mathcal{P}_\mathbb{R}\times\mathbb{N}^6\ra\mathcal{P}_\mathbb{R}$ and $\pi_2:\mathcal{P}_\mathbb{R}\times\mathbb{N}^6\ra \mathbb{N}^6$ be the canonical projections.
\par
We claim that $\Gamma$ is not saturated: first notice that $\pi_2\circ\pi_1^{-1}(\varpi_2)=S(\varpi_2)$. Pick a lattice point $\bold{q}\in P$ which is not in $S(\varpi_2)$. Since $q\in\textrm{conv}(S(\varpi_2))$, there exists $s_1,\ldots,s_m\in\mathbb{Q}$ and $\bold{p}_1,\ldots,\bold{p}_m\in S(\varpi_2)$ such that $s_1+\ldots+s_m=1$ and $\bold{q}=s_1\bold{p}_1+\ldots+s_m\bold{p}_m$. Multiplying both sides by the least common multiple $M$ of the denominators of $s_1,\ldots,s_m$, we know that $(M\varpi_2,M\bold{q})\in \Gamma$ as $\Gamma$ is a monoid. If $\Gamma$ were saturated, $(M\varpi_2,M\bold{q})\in \Gamma$ will imply $(\varpi_2,\bold{q})\in\Gamma$, contradicts to $\pi_2\circ\pi_1^{-1}(\varpi_2)=S(\varpi_2)$.
\par
This example explains that the saturated assumption in \cite{FaFoL15} is necessary.
\end{remark}

\section{Various quantum degree cones}\label{sec6}

\subsection{Lie algebra $\tt G_2$}\label{Sec:G2}

Let $\lie g$ be the Lie algebra of type $\tt G_2$ with positive roots 
$$\Delta_+=\{\alpha_1,\alpha_2,\alpha_1+\alpha_2,2\alpha_1+\alpha_2,3\alpha_1+\alpha_2,3\alpha_1+2\alpha_2\}.$$
For $\bold{d}\in \mathbb{R}^{\Delta_+}_{\geq 0}$, we write $d_1=\bold{d}(\alpha_1)$, $d_2=\bold{d}(\alpha_2)$, $d_{12}=\bold{d}(\alpha_1+\alpha_2)$, $d_{112}=\bold{d}(2\alpha_1+\alpha_2)$, $d_{1112}=\bold{d}(3\alpha_1+\alpha_2)$ and $d_{11122}=\bold{d}(3\alpha_1+2\alpha_2)$. 
The classical degree cone $\md\subset \mathbb{R}^{\Delta_+}_{\geq 0}$ is determined by the following inequalities:
$$ d_1 + d_2 > d_{12},\ \ d_{1} + d_{12} > d_{112}, \ \ d_{1} + d_{112} > d_{1112}, $$
$$ d_2 + d_{1112} > d_{11122},\ \ d_{112} + d_{12} > d_{11122}.$$
For example $(d_{1}, d_{1112},d_{112},d_{11122},d_{12},d_2)=(2,1,3,1,3,2) \in \mathcal{D}$. 
\par
We fix a reduced decomposition $\underline{w}_0=s_1s_2s_1s_2s_1s_2\in R(w_0)$. Let 
$$F_{1},\ \ F_{1112},\ \ F_{112},\ \ F_{11122},\ \ F_{12},\ \ F_2$$ 
be the corresponding quantum PBW root vectors. 
The quantum degree cone $\Dwq$ in $\mathcal{D}$ is given by the following inequalities:
\begin{equation}\label{Eq:G2}
\begin{aligned}
d_1+d_{11122}>2d_{112},\ \ d_{1112} + d_{11122} > 3d_{112},\ \ d_{1112} + d_{12} > 2d_{112},\\
d_{1112} + d_{2} > d_{112} + d_{12},\ \ d_{112} + d_{2} > 2d_{12},\ \ d_{11122} + d_{2} > 3 d_{12}.
\end{aligned}
\end{equation}
These inequalities are the same for the other reduced decomposition $\underline{w}_0'=s_2s_1s_2s_1s_2s_1\in R(w_0)$. It is clear that $\bold{d}\in\mathcal{D}_{\underline{w}_0}^q$ for any $\underline{w}_0\in R(w_0)$.

\subsection{Lie algebra $\tt A_3$}\label{Sec:A3}
Let $\g$ be of type $\tt A_3$. For $\bold{d}\in\mathbb{R}_{\geq 0}^{\Delta_+}$, we write $d_{i,j}=\bold{d}(\alpha_{i,j})$. The classical degree cone $\mathcal{D}\subset\mathbb{R}_{\geq 0}^{\Delta_+}$ is defined by the following inequalities:
$$d_{1,1} + d_{2,2} > d_{1,2},\ \ d_{2,2} + d_{3,3} > d_{2,3},\ \ d_{1,1} + d_{2,3} > d_{1,3},\ \ d_{1,2} + d_{3,3} > d_{1,3}, $$

Let $\underline{w}^1_0 = s_1s_2s_1s_3s_2s_1$ and $\underline{w}^2_0 = s_1s_3s_2s_3s_1s_2\in R(w_0)$. We claim that the corresponding quantum degree cones satisfy $\mathcal{D}^q_{\!\!\underline{w}^1_0} \cap \mathcal{D}^q_{\!\!\underline{w}^2_0} = \emptyset$.
\par
Let $F_{i,j}$ (resp. $F_{i,j}'$) denote the quantum PBW root vector associated to $\underline{w}_0^1$ (resp. $\underline{w}_0^2$) and root $\alpha_{i,j}$. We have the following commutation relations between the quantum PBW root vectors:
$$F_{1,2} F_{2,3} = F_{2,3} F_{1,2} + (q-q^{-1}) F_{2,2} F_{1,3},\ \ F'_{1,3} F'_{2,2} = F'_{2,2} F'_{1,3} + (q-q^{-1}) F'_{1,2}F'_{2,3},$$
giving two contradict inequalities in the corresponding quantum degree cones:
$$d_{1,2} + d_{2,3} > d_{2,2} + d_{1,3},\ \ d_{1,3} + d_{2,2} > d_{1,2} + d_{2,3}.$$
Projecting to the corresponding coordinates proves the claim.

\subsection{Lie algebra $\tt B_3$}\label{Sec:B3}
Let $\g$ be of type $\tt B_3$. The set of positive roots 
$$\Delta_+=\{\alpha_{1,1},\alpha_{1,2},\alpha_{1,3},\alpha_{2,2},\alpha_{2,3},\alpha_{3,3}, \alpha_{1,\overline{3}}, \alpha_{2,\overline{3}}, \alpha_{1,\overline{2}}\}.$$ 
For $\bold{d}\in\mathbb{R}_{\geq 0}^{\Delta_+}$, we write $d_{i,j}=\bold{d}(\alpha_{i,j})$ and $d_{i,\overline{j}}=\bold{d}(\alpha_{i,\overline{j}})$. The classical degree cone $\mathcal{D}$ is determined by:
$$d_{1,1} + d_{2,2} > d_{1,2},\ \ d_{1,1} + d_{2,3} > d_{1,3},\ \ d_{1,1} + d_{2,\overline{3}} > d_{1,\overline{3}},\ \ d_{1,2} + d_{3,3} > d_{1,3},$$
$$d_{1,2} + d_{2,\overline{3}} > d_{1,\overline{2}},\ \  d_{2,2} + d_{3,3} > d_{2,3},\ \ d_{2,2} + d_{1,\overline{3}} > d_{1,\overline{2}},$$
$$d_{1,3} + d_{2,3} > d_{1,\overline{2}},\ \ d_{1,3} + d_{3,3} > d_{1,\overline{3}},\ \ d_{2,3} + d_{3,3} > d_{2,\overline{3}}.$$
For example $\bd = (4,3,3,3,1,1,4,3,2) \in \mathcal{D}$.
\par
We consider 
$$\underline{w}_0^1 = s_1s_2s_1s_3s_2s_1s_3s_2s_3\ \ \text{and}\ \ \underline{w}^2_0 = s_1s_3s_2s_3s_2s_1s_2s_3s_2\in R(w_0).$$
The quantum degree cone $\mathcal{D}_{\!\!\underline{w}^1_0}^q$ in $\mathcal{D}$ is defined by the following inequalities:
$$d_{1,1} + d_{1,\overline{2}} > 2d_{1,3},\ \ d_{1,2} + d_{1,\overline{2}} > d_{2,2} + 2d_{1,3},\ \ d_{1,2} + d_{1,\overline{3}} > 2d_{1,3},\ \ d_{1,2} + d_{2,\overline{3}} > d_{2,2} + d_{1,\overline{3}},$$
$$d_{1,2} + d_{2,\overline{3}} > d_{1,3} + d_{1,2},\ \ d_{2,2} + d_{2,\overline{3}} > 2 d_{2,3},\ \ d_{1,3} + d_{2,\overline{3}} > d_{1,\overline{3}} + d_{2,3}, $$ 
$$d_{1,\overline{2}} + d_{2,\overline{3}} > d_{1,\overline{3}} + 2d_{2,3},\ \ d_{1,\overline{2}} + d_{3,3} > d_{1,\overline{3}} + d_{2,3},\ \ d_{1,2} + d_{2,3} > d_{1,3} + d_{2,2}.$$
The quantum degree cone $\mathcal{D}_{\!\!\underline{w}^2_0}^q$ in $\mathcal{D}$ is defined by the following inequalities:
$$d_{1,1} + d_{1,\overline{2}} > 2d_{1,3},\ \ d_{1,1} + d_{1,\overline{2}} > d_{1,2} + d_{1,\overline{3}},\ \ d_{1,2} + d_{1,\overline{3}} > 2d_{1,3},\ \ d_{1,\overline{3}} + d_{2,3} > d_{1,3} + d_{2,\overline{3}},$$
$$ d_{1,\overline{3}} + d_{2,2} > d_{1,3} + d_{2,3},\ \ d_{1,\overline{3}} + d_{2,2} > d_{1,2} + d_{2,\overline{3}},\ \ d_{1,3} + d_{2,3} > d_{1,2} + d_{2,\overline{3}}, $$
$$d_{2,\overline{3}} + d_{2,2} > 2d_{2,3},\ \ d_{1,3} + d_{2,2} > d_{1,2} + d_{2,3}.$$
By the contradiction of the last inequalities, we obtain $\mathcal{D}^q_{\!\!\underline{w}^1_0} \cap \mathcal{D}^q_{\!\!\underline{w}^2_0} = \emptyset$.

\subsection{Lie algebra $\tt C_3$}\label{Sec:C3}

Let $\lie g$ be of type $\tt C_3$. For $\bold{d}\in\mathbb{R}_{\geq 0}^{\Delta_+}$, we write $d_1=\bold{d}(\alpha_{1,1})$, $d_2=\bold{d}(\alpha_{1,2})$, $d_3=\bold{d}(\alpha_{1,\overline{1}})$, $d_4=\bold{d}(\alpha_{1,3})$, $d_5=\bold{d}(\alpha_{1,\overline{2}})$, $d_6=\bold{d}(\alpha_{2,2})$, $d_7=\bold{d}(\alpha_{2,\overline{2}})$, 
$d_8=\bold{d}(\alpha_{2,3})$ and $d_9=\bold{d}(\alpha_{3,3})$.
The classical degree cone $\mathcal{D}\subset \mathbb{R}_{\geq 0}^{\Delta_+}$ is determined by:
$$d_1 + d_5 > d_3,\ \ d_1 + d_6 > d_2,\ \ d_1 + d_7 > d_5,\ \ d_1 + d_8 > d_4,\ \ d_2 + d_4 > d_3,$$
$$d_2 + d_8 > d_5,\ \ d_2 + d_9 > d_4,\ \ d_4+d_6 > d_5,\ \ d_6 + d_8 > d_7,\ \ d_6+d_9 > d_8.$$
We consider the reduced decompositions 
$$\underline{w}^1_0=s_1s_2s_3s_2s_1s_2s_3s_2s_3\ \ \text{and}\ \ \underline{w}^2_0=s_1s_3s_2s_3s_2s_1s_2s_3s_2.$$

Moreover, the inequalities determining the cone $\mathcal{D}^q_{\!\!\underline{w}^1_0}$ in $\mathcal{D}$ are:
$$d_1 + d_5 > d_2 + d_4,\ \ d_3+d_9 > 2d_4,\ \ d_7 + d_9 > 2d_8,\ \ d_3 + d_7 > 2d_5,$$
$$d_1 + d_7 > d_4 + d_6,\ \ d_2+d_7 > d_5 + d_6,\ \ d_2 + d_7 > d_4 + 2d_6,\ \ d_3+d_7 > d_4+d_5+d_6,$$
$$d_3+d_7> 2d_4+2d_6,\ \ d_3+d_8 > d_4 + d_5,\ \ d_3+d_8 >2d_4 + d_6,\ \ d_2 + d_8 > d_4 + d_6. $$

The inequalities determining the cone $\mathcal{D}^q_{\!\!\underline{w}^2_0}$ in $\mathcal{D}$ are:
$$d_1 + d_5 > d_2 + d_4,\ \ d_3 + d_9 > 2d_4,\ \ d_7 + d_9 > 2d_8,\ \ d_3 + d_7 > 2d_5,$$
$$d_1 + d_7 > d_2 + d_8,\ \ d_4 + d_7 > d_5 + d_8,\ \ d_3 + d_7 > d_2 + d_5 + d_8 ,\ \ d_3 + d_7 > 2d_2 + 2d_8 ,,$$
$$d_3 + d_6 > d_2 + d_5,\ \ d_3 + d_6 > 2d_2 + d_8,\ \ d_4 + d_7 > d_2 + d_8,\ \ d_4 + d_6 > d_2 + d_8.$$
Notice that there is a contradiction in the last inequalities, implying that $\mathcal{D}^q_{\!\!\underline{w}^1_0} \cap \mathcal{D}^q_{\!\!\underline{w}^2_0} = \emptyset$.

There are four elements in $\mathcal{D}^q_{\!\!\underline{w}^1_0 }$, which are minimal regarding the sum over all entries:
$$ \mathbf{d}_1 = (2,1,1,1,1,1,4,4,5),\ \mathbf{d}_2 = (3,2,2,1,1,1,3,3,4),$$
$$ \mathbf{d}_3 = (5,4,4,1,1,1,1,1,2),\ \mathbf{d}_4 = (4,3,3,1,1,1,2,2,3).$$
Since $\mathbf{d}_1, \mathbf{d}_2, \mathbf{d}_3 ,\mathbf{d}_4 \in \mathcal{D}$ we go back to the classical case. We consider the fundamental module $V(\varpi_2)$ and the weight 
$\tau = 2\alpha_1 + 3\alpha_2 + \alpha_3$
 whose weight space $V(\varpi_2)_{\varpi_2 - \tau}$ is of dimension 1. We have to choose an element with minimal degree from the following set, where we neglect the elements which have obviously a higher degree:
$$\{f_{1,2}f_{1,\overline{2}}, f_{1,\overline{1}}f_{2,2}\}.$$
For each of the above elements in $\mathcal{D}^q_{\underline{w}_0}$ both monomials have the same degree, so we do not obtain a monomial ideal $I^{\bd_i}$, $1 \leq i \leq 4$.\par
By taking larger degrees $\bd \in \Dwq$ it is possible to obtain a unique monomial basis of $V^{\bd}(\varpi_2)$, where it is possible to obtain a basis with either of both monomials applied to $v^{\bd}_{\varpi_2}$. We conclude $\mathcal{D}^q_{\underline{w}_0} \nsubseteq \mathcal{S}_{\lm}$, but $\mathcal{D}^q_{\underline{w}_0} \cap \mathcal{S}_{\lm} \neq \emptyset$. We also see, different elements in $\mathcal{D}^q_{\underline{w}_0 }$ can produce different monomial bases. This observation still holds, even if we consider elements where the sum over the entries is the same.\par

\subsection{Lie algebra $\tt D_4$}\label{Sec:D4}
Let $\lie g$ be of type $\tt D_4$. In the Dynkin diagram we let $2$ be the central node. We consider the following reduced decomposition 
$$\underline{w}_0 = s_2s_1s_2s_3s_2s_4s_2s_1s_2s_3s_2s_4\in R(w_0)$$
For a positive root $a \alpha_1 + b \alpha_2 + c \alpha_3 + d \alpha_4\in\Delta_+$, we let $f_{abcd}$ denote the corresponding quantum PBW root vector. In the convex order on positive roots given by $\underline{w}_0$, they are:
$$f_{0100},\ \ f_{1100},\ \ f_{1000},\ \ f_{1110}, \ \ f_{0110},\ \ f_{1211},\ \ f_{1101},\ \ f_{1111},\ \ f_{0010},\ \ f_{0111},\ \ f_{0101},\ \ f_{0001}.$$
For $d\in\mathbb{R}_{\geq 0}^{\Delta_+}$, let $d_i$ be the value of $\bold{d}$ at the positive root corresponding to the $i$-th quantum PBW root vector above. The quantum degree cone $\Dwq\subset \mathbb{R}_{\geq 0}^{\Delta_+}$ is defined by:
$$d_1 + d_3 > d_2,\ \ d_1 + d_8 > d_5 + d_7,\ \ d_1 + d_8 > d_6,\ \ d_1 + d_9 > d_5,\ \ d_1 + d_{12} > d_{11}$$
$$d_2 + d_8 > d_3+d_5+d_7 ,\ \ d_2 + d_8 > d_3 + d_6 ,\ \ d_2 + d_8 > d_4 + d_7 ,\ \ d_2 + d_9 > d_3 + d_5,$$
$$d_2 + d_9 > d_4,\ \ d_2 + d_{10} > d_6,\ \ d_2 + d_{12} > d_3 + d_{11},\ \ d_2 + d_{12} > d_7$$
$$d_3 + d_5 > d_4,\ \ d_3 + d_{10} > d_7 + d_9 ,\ \ d_3 + d_{10} > d_8, \ \ d_3 + d_{11} > d_7   $$
$$ d_4 + d_{10} > d_5 + d_7 +d_9 ,\ \ d_4 + d_{10} > d_5 + d_8 ,\ \ d_4 + d_{10} > d_6 + d_9$$
$$ d_4 + d_{11} > d_5 + d_7 ,\ \ d_4 + d_{11} > d_6 ,\ \ d_4 + d_{12} > d_8,\ \  d_5 + d_7 > d_6,$$
$$d_5 + d_{12} > d_9 + d_{11},\ \ d_5 + d_{12} > d_{10},\ \ d_6 + d_{12} > d_7 + d_9 + d_{11},\ \  d_6 + d_{12} > d_7 + d_{10}$$
$$d_6 + d_{12} > d_8 + d_{11},\ \ d_7 + d_9 > d_8,\ \ d_9 + d_{11} > d_{10}. $$
For example, $\bold{d} = (5,5,1,2,4,1,1,2,6,10,12,20)\in \Dwq$.


\end{document}